\documentclass[11pt]{article}
\usepackage{amsmath,amssymb,amsthm}
\usepackage{authblk}
\usepackage{natbib}
\usepackage{comment}
\usepackage{makeidx}
\usepackage[all]{xy}           %For commutative diagrams
\usepackage{amsfonts}
\usepackage{mathtools}
\usepackage{commath}
\usepackage[sc,osf]{mathpazo}
\usepackage{eucal}
\usepackage{calc}
\usepackage{graphicx}
\usepackage[margin=1cm]{caption}
\usepackage{multicol}
\usepackage{tikz}
\usepackage{amsmath}
\usetikzlibrary{arrows.meta, decorations.pathreplacing, calc, shapes.misc}
\usepackage{pgfplots}
\pgfplotsset{compat=1.18}
\usepackage[hidelinks]{hyperref}

\newcommand{\e}{\mbox{exp}}

\newcommand{\ca}{\mbox{cay}}
\newcommand{\ad}{\mbox{ad}}

\newcommand{\al}{\mathfrak{g}}

\numberwithin{equation}{section}

\newtheorem{definition}{Definition}[section]

\newtheorem{theorem}[definition]{Theorem}
\newtheorem{proposition}[definition]{Proposition}

\newtheorem{remarkth}[definition]{Remark}
\newtheorem{example}[definition]{Example}
\newenvironment{remark}{\begin{remarkth}\upshape}{\hfill$\diamond$\end{remarkth}}
\renewcommand{\emph}[1]{{\bfseries\itshape{#1}}}
\usepackage{amssymb}
\newcommand{\R}{\mathbb{R}}      %Numeros reales
\newcount\ancho \newcount\anchom \newcount\anchoa
\newcount\anchob \newcount\altura

\newcommand{\ltilde}[3][0]{\altura=0 \advance\altura by #1
	\ancho=#2 \anchom=\ancho \divide\anchom by 2
	\anchoa=\ancho \divide\anchoa by 4
	\anchob=\anchom \advance\anchob by \anchoa
	\kern-3pt \begin{array}[b]{c}
		\begin{picture}(1,1)(\anchom,-\altura)
		\qbezier(0,2)(\anchoa,5)(\anchom,2)
		\qbezier(\anchom,2)(\anchob,-1)(\ancho,4)
		\qbezier(0,2)(\anchoa,4.5)(\anchom,1.8)
		\qbezier(\anchom,1.8)(\anchob,-1.5)(\ancho,4)
		\end{picture} \\[-4pt]{#3}
	\end{array} \kern-4pt    }

\newcommand{\lhat}[3][0]{\altura=0 \advance\altura by #1
	\ancho=#2 \anchom=\ancho \divide\anchom by 2
	\anchoa=\ancho \divide\anchoa by 4
	\anchob=\anchom \advance\anchob by \anchoa
	\kern-3pt \begin{array}[b]{c}
		\begin{picture}(1,1)(\anchom,-\altura)
		\qbezier(0,2)(\anchoa,4)(\anchom,6)
		\qbezier(\anchom,6)(\anchob,4)(\ancho,2)
		\qbezier(0,2)(\anchoa,3.8)(\anchom,5.6)
		\qbezier(\anchom,5.6)(\anchob,3.8)(\ancho,2)
		\end{picture} \\[-4pt] {#3}
	\end{array} \kern-4pt    }

\newcommand{\lvec}[1]{\overleftarrow{#1}}
\newcommand{\rvec}[1]{\overrightarrow{#1}}

\newcommand{\tr}{\operatorname{Tr}}
\makeatletter
\newcommand\prol{\@ifstar{\@proldf}{\@prolpf}}  %% if * dual else primal
\def\@prolpf{\@ifnextchar[{\@prolpf@wrt}{\@prolpf@}}
\def\@prolpf@wrt[#1]#2{\@ifnextchar[{\@prolpf@wrt@at{#1}{#2}}{\@prolpf@wrt@{#1}{#2}}}
\def\@prolpf@wrt@at#1#2[#3]{\prolsymbol^{#1}_{#3}#2}
\def\@prolpf@wrt@#1#2{\prolsymbol^{#1}#2}
\def\@prolpf@#1{\@ifnextchar[{\@prolpf@at{#1}}{\@prolpf@@{#1}}}
\def\@prolpf@at#1[#2]{\prolsymbol_{#2}#1}
\def\@prolpf@@#1{\prolsymbol#1}
\def\@proldf{\@ifnextchar[{\@proldf@wrt}{\@proldf@}}
\def\@proldf@wrt[#1]#2{\@ifnextchar[{\@proldf@wrt@at{#1}{#2}}{\@proldf@wrt@{#1}{#2}}}
\def\@proldf@wrt@at#1#2[#3]{\prolsymbol^{*#1}_{#3}#2}
\def\@proldf@wrt@#1#2{\prolsymbol^{*#1}#2}
\def\@proldf@#1{\@ifnextchar[{\@proldf@at{#1}}{\@proldf@@{#1}}}
\def\@proldf@at#1[#2]{\prolsymbol^*_{#2}#1}
\def\@proldf@@#1{\prolsymbol^*#1}
\def\prolsymbol{\mathcal{T}}
\makeatother

\begin{document}

\title{Discrete variational calculus\\ for double-bracket dissipation}

 \author[1]{Anthony Bloch} 
 \author[2]{Sebastián J. Ferraro}
\author[3]{David Mart{\'\i}n de 
Diego}
\author[4]{Shreyas Bharadwaj}
 \affil[1]{Department of Mathematics, University of Michigan, abloch@umich.edu}
 \affil[2]{
 Departamento de Matemática, Universidad Nacional del Sur (UNS)--CONICET, Bahía Blanca, Argentina, 
 sferraro@uns.edu.ar}
 \affil[3]{Instituto de Ciencias Matemáticas (ICMAT), Spain, david.martin@icmat.es}
 \affil[4]{Indian Institute of Technology Bombay, India, shreyasnb@iitb.ac.in}

\maketitle
\begin{abstract}
Discrete variational methods show excellent performance in numerical simulations of mechanical systems. In this paper, we adapt discrete variational integrators for the case of mechanical systems with double-bracket dissipation. In particular, we will work with forced  Euler--Poincar\'e and forced Lie--Poisson systems, and the case of interest for us will be when the coadjoint orbits remain invariant, but  the energy is decreasing along the orbit. This particular kind of dissipative system appears in various physical systems such as  satellites with dampers, geophysical fluids, plasma physics and  stellar dynamics. The proposed geometric integrator preserves the coadjoint orbits exactly. We highlight the advantages of this feature by comparing it with other general-purpose methods (including higher-order ones) across different numerical simulations.
\end{abstract}
{
  \setlength{\parskip}{0pt} % Removes extra paragraph spacing
  \tableofcontents
}
\section{Introduction}
For the numerical study of differential equations it is important to preserve the geometrical properties of the continuous flow as much as possible. For instance, one can aim to preserve the configuration manifold, first integrals, symplecticity, reversibility or volume. This is crucial for a qualitatively and quantitatively correct simulation of the continuous system. These types of integrators are called geometric integrators (see \cite{hairer,blanes, serna} and references therein). 
A particular case, variational integrators, which arise from a discrete variational principle, is one of a class of geometric integrators which are easily adapted to the numerical integration of different classes of mechanical systems \citep{marsden-west}. This makes these methods very useful for real-world applications as well as for other more general mechanical systems.

An interesting example is the case of Lagrangian systems subjected to external forces, including dissipative ones. Some of these systems have very interesting additional geometric features, and the idea of this paper is to study the possibility of preserving them using a modification of variational integrator techniques. In particular, we will work with forced  Euler--Poincar\'e and forced Lie--Poisson systems. A particular case  of interest for us will be when the underlying coadjoint orbit on which the system evolves remains invariant, but the energy is decreasing along orbits, since we are dealing with dissipative systems (see \cite{bloch1996euler}). This particular kind of dissipative system appears in different physical systems such as satellites with dampers, geophysical fluids, plasma physics and stellar dynamics.  
The final aim of this paper is to introduce  a new type of variational integrator preserving exactly the coadjoint orbits for systems with double-bracket dissipation. 
In \cite{bloch1996euler} the authors study the perturbation of the standard Lie--Poisson equations adding a particular dissipative term that preserves the coadjoint orbits. Such systems are typically expressed on the dual of a Lie algebra using the Lie--Poisson bracket modified with a symmetric bracket which introduces the dissipation. 
The metriplectic setting discussed in 
\cite{morrison} and \cite{Bloch-Morrison-Ratiu} is an interesting and instructive related or complementary  case (see also \cite{BFM2025}).
The rigid body serves as a running example throughout the paper. The dissipative forces are a simple model for internal friction, which does not exert any net torque on the body.

The paper is structured as follows. In Section~\ref{Section2}, we recall the Euler–Poincaré equations on a Lie group and their Hamiltonian formulation (the Lie–Poisson equations), highlighting the associated geometric invariants, such as coadjoint orbits and energy preservation. When external forces are introduced, these geometric invariants are, in general, lost. However, in Section~\ref{sectio3}, we describe a particular class of external forces that preserves the coadjoint orbits while dissipating energy. These forces are known as double-bracket dissipation.
After introducing this form of double-bracket dissipation via a constrained variational principle, Section~\ref{sec:Discrete_Lagrangian_formalism_for_DBD} presents a geometric integrator that exactly preserves the coadjoint orbits. Then, using retraction maps from the Lie algebra to the Lie group, we construct several geometric integrators adapted to different retractions (see \cite{Bou-Rabee2009}). Finally, in Section~\ref{sec5}, we compare our methods with general-purpose integrators, demonstrating their excellent performance and convergence to asymptotically stable equilibria.

\section{Euler--Poincar\'e and Lie--Poisson equations with forcing}\label{Section2}

\subsection{Euler--Poincar\'e equations with forcing}

Let $G$ be a Lie group. 
Left multiplication ${\mathcal L}_g$ allows us to trivialize the tangent bundle $TG$ and the cotangent bundle $T^*G$ as follows (see \cite{Abraham1978,marsden2013introduction,holm-book1,holm-book2}):
\begin{eqnarray*}
	TG&\to&G\times {\mathfrak g}\, ,\qquad (g, \dot{g})\longmapsto (g, g^{-1}\dot{g})=(g, T_g{\mathcal L}_{g^{-1}}\dot g)=(g, \xi)\; ,\\
	T^*G&\to&G\times {\mathfrak g}^*,\qquad  (g, \mu_g)\longmapsto (g, T^*_e {\mathcal L}_g(\mu_g))=(g, \mu)\; ,
\end{eqnarray*}
where ${\mathfrak g}=T_eG$ is the Lie algebra of $G$ and $e$ is the identity element in $G$. In this paper, we will work mainly with the left multiplication. Of course, it is possible to derive similar results using the right multiplication.

Given a Lagrangian $L\colon TG\rightarrow {\mathbb R}$,  after left trivialization, the classical Euler--Lagrange equations may be  rewritten as 
\begin{align*}
\frac{d}{dt}\left(\frac{ \delta L}{\delta \xi}\right)&=ad_{\xi}^*\frac{\delta L}{\delta\xi}+T^*_e{\mathcal L}_g\left(\frac{\delta L}{\delta g}\right)\\%\label{elt}\\
\dot{g}&=g\xi\,.
\end{align*}
We will assume, in the sequel, that  $L\colon TG\rightarrow {\mathbb R}$ is left invariant, meaning that $L(g,g\xi)=L(e,\xi)$ for all $g\in G$, $\xi\in\mathfrak{g}$; that is, the following reduced Lagrangian $l\colon {\mathfrak g}\rightarrow {\mathbb R}$: 
\[
l(\xi)=L(g, g\xi)=L(e,\xi)
\]
is well defined. 
Observe that $l$ can be seen as the  restriction of $L$ to ${\mathfrak g}=T_eG$. In this case, the corresponding Euler--Lagrange equations are: 
\begin{equation}\label{awr}
\frac{d}{dt}\left(\frac{ \delta l}{\delta \xi}\right)=ad_{\xi}^*\frac{\delta l}{\delta\xi}
\end{equation}
and they are known as the (left-invariant) {\bf Euler--Poincar\'e equations}.

To find a solution $t\mapsto g(t)$ of the Euler--Lagrange equations, with initial conditions $g(0)=g_0$ and $\dot{g}(0)=v_0$,  we first  solve the first order differential equation defined by the Euler--Poincar\'e equations: 
\[
\frac{d}{dt}\left(\frac{ \delta l}{\delta \xi}\right)=ad_{\xi}^*\frac{\delta l}{\delta\xi}
\]
where now the  initial condition  is $\xi(0)=g_0^{-1}v_0$ and, finally, with this solution $t\rightarrow \xi(t)$ we solve the so-called reconstruction equation: 
\[
\dot{g}(t)=g(t)\xi(t) \hbox{   with   } g(0)=g_0\; .
\]

The Euler--Poincar\'e equations are modified under the presence of forces that, in this paper,   are expressed mathematically as a map $f\colon {\mathfrak g}\rightarrow {\mathfrak g}^*$. The Euler--Poincar\'e equations with forcing are  
\[
\frac{d}{dt}\left(\frac{ \delta l}{\delta \xi}\right)=ad_{\xi}^*\frac{\delta l}{\delta\xi}+f(\xi)
\]
If $G$ is finite dimensional and we fix a basis $\{e_a\}$ of the Lie algebra ${\mathfrak g}$, inducing coordinates $(\xi^a)$ such that $\xi=\xi^a e_a$, then the Euler--Poincar\'e equations with forcing have the following expression in local coordinates:
\[
\frac{d}{dt}\left(\frac{ \delta l}{\delta \xi^a}\right)=C^c_{ba}\xi^b \frac{\partial l}{\partial \xi^c} +f_a
\]
where $C_{ab}^c$ are the structure constants of the Lie algebra ${\mathfrak g}$ defined by
$[e_a, e_b]=C_{ab}^c e_c$ 
 and $\langle f(\xi), e_a\rangle=f_a$.

\

\subsection{Lie--Poisson equations with forcing}

Given a Hamiltonian  $H\colon {\mathfrak g}^*\rightarrow {\mathbb R}$, the Lie--Poisson equations are
\begin{align}
\dot{\mu}&=ad_{{\delta H}/{\delta \mu}}^*\mu\label{elth}\\
\dot{g}&= g \frac{\delta H}{\delta \mu}=T_e{\mathcal L}_g\left(\frac{\delta H}{\delta \mu}\right)\, , \notag
\end{align}
where $\mu\in {\mathfrak g}^*$.

Of course, Equations \eqref{awr} and \eqref{elth} are related by the Legendre transformation when the Lagrangian $l$ is regular. This Legendre transformation is given by
\[
\begin{array}{rrcl}
{\mathcal F}l\colon &{\mathfrak g}&\longrightarrow&  {\mathfrak g}^*\\
&  \xi&\longmapsto & {\delta l}/{\delta \xi}
\end{array}
\]
Assuming that  $l$ is hyperregular, that is,  the Legendre map is a diffeomorphism, then there exists a unique $\xi$ such that $\mu=\frac{\partial l}{\partial \xi}(\xi)$, for each $\mu \in {\mathfrak g}^*$. Then the Hamiltonian is defined by 
\[
H(\mu)=\langle \mu, \xi\rangle - l(\xi)\; .
\]

%It is well known that both brackets are induced by reduction of the standard Lie bracket on $T^*G$ by right or left-%translation.
We note that ${\mathfrak g}^*$  is equipped with the Lie--Poisson bracket $\{\cdot, \cdot \}$:
\[
\{f, g\}(\mu)=-\left\langle \mu, \left[\frac{\delta f}{\delta \mu}, \frac{\delta g}{\delta \mu}\right] \right\rangle, \qquad f, g\in C^{\infty}({\mathfrak g}^*)\; .
\]
This bracket  exactly corresponds to the reduced bracket by standard  Poisson reduction \cite{marsden-ratiu} from
\[
\pi\colon (T^*G, \omega_G)\longrightarrow (T^*G/G\equiv {\mathfrak g}^*, \{\cdot , \cdot \})
\]
where $\pi(\mu_g)=[\mu_g]\equiv T_e^*{\mathcal L}_g(\mu_g)$.

The coadjoint orbit through a fixed  $\mu_0\in {\mathfrak g}^*$ is
\[
{\mathcal O}_{\mu_0}:=\left\{ Ad^*_{g^{-1}}\mu_0\; |\; g\in G\right\}\subseteq {\mathfrak g}^*\; .
\]
If $t\rightarrow \mu(t)$ is the solution of  the initial value problem
$\dot{\mu}=ad_{{\delta H}/{\delta \mu}}^*\mu$ with $\mu(0)=\mu_0$, then  we can deduce that 
$
\mu(t)\in {\mathcal O}_{\mu(0)}
$.
A generalized Casimir invariant (or Casimir for short) is a function $C\colon {\mathfrak g}^*\rightarrow {\mathbb R}$ such that 
$\{C, f\}=0$ for all $f\colon {\mathfrak g}^*\rightarrow {\mathbb R}$ or, using the definition of Lie--Poisson bracket: 
\[
\left[\frac{\partial C}{\partial \mu}, \xi\right]=0, \quad \forall \xi \in {\mathfrak  g}\; . 
\]

Given a Hamiltonian function $H\colon {\mathfrak g}^*\rightarrow {\mathbb R}$  the equations of motion are determined by 
\begin{equation}\label{plo}
\dot{f}=\{f, H\}=
%\Pi^{\sharp}(dH(\mu(t)))=
X_H(f)=\left(\frac{\partial f}{\partial \mu}\right)^T\Pi
\left(\frac{\partial H}{\partial \mu}\right),
\end{equation}
where $\Pi$ is the bivector field associated to the bracket $\{\cdot ,\cdot \}$. In coordinates, 
\[
\Pi(\mu)=(\Pi_{ab})\left( \mu_cC_{ba}^c\right),
\]
and $X_H$ is the corresponding Hamiltonian vector field:
\[
X_H=\mu_c C^c_{ab}\frac{\partial H}{\partial \mu^a}\frac{\partial}{\partial \mu^b}.
\]
It is well known that the flow $\Psi^H_t\colon {\mathfrak g}^*\rightarrow {\mathfrak g}^*$  of $X_H$ satisfies some geometric properties: 
\begin{enumerate}
	\item It preserves the linear Poisson bracket, that is
	\[
	\{f\circ \Psi_t, g\circ \Psi_t\}=\{f, g\}\circ \Psi_t, \qquad f, g\in C^{\infty}({\mathfrak g}^*)\,.
	\]
	\item It preserves the Hamiltonian
	\[
	H\circ \Psi_t=H\,.
	\]
	\item If all the coadjoint orbits are connected, Casimir functions are also preserved along each coadjoint orbit. 
\end{enumerate}

Now we add forces to our picture, in this case modeled as a map $\tilde{f}\colon {\mathfrak g}^*\rightarrow {\mathfrak g}^*$. 
If we start from a force on the Lagrangian side $f\colon {\mathfrak g}\rightarrow {\mathfrak g}^*$ then we define the force on the Hamiltonian side by taking 
$\tilde{f}\circ {\mathcal F}l=f$, in the regular case. The Lie--Poisson equations with forcing are modified as follows:
\[
\dot{\mu}=ad_{{\delta H}/{\delta \mu}}^*\mu +\tilde{f}(\mu).\; \\
\]
It is clear that in adding the force term we are losing, in general, the key  properties of the flow of the free system (preservation of the Hamiltonian, preservation of coadjoint orbits, etc.). But, in some cases of interest, as we will show in the next section, it is possible to derive forced systems dissipating energy but preserving the coadjoint orbits \citep{bloch1996euler}. 

\section{Double-bracket dissipation}\label{sectio3}

Consider now a force $f\colon {\mathfrak g}\rightarrow {\mathfrak g}^*$ with a particular form \citep{bloch1996euler}:
\[
f(\xi)=ad^*_{\varphi(\xi)} \frac{\delta l}{\delta \xi},
\]
where $\varphi\colon {\mathfrak g}\rightarrow {\mathfrak g}$ is an arbitrary map. Then, the forced Euler--Lagrange equations are: 
\begin{equation}\label{fel-1}
\frac{d}{dt}\left(\frac{\delta l}{ \delta \xi}\right)=ad_{\xi+\varphi(\xi)}^*	\frac{\delta l}{ \delta \xi}
\end{equation}

In the finite dimensional case we can write the previous equations in coordinates as follows:
\begin{equation}\label{fel2-1}
\frac{d}{dt}\left(\frac{\delta l}{ \delta \xi^b}\right)=C_{ab}^d (\xi^a+\varphi^a(\xi)) \frac{\delta l}{ \delta \xi^d}
\end{equation}

\begin{example}{Rigid body with a force determined by a map $\varphi\colon {\mathfrak g}\rightarrow {\mathfrak g}$}

{\rm 
In the case of ${\mathfrak g}={\mathbb R}^3$ with the usual vector cross product we have
\[
\frac{d}{dt}\left(\frac{\delta l}{ \delta {\mathbf \Omega}}\right)=\frac{\delta l}{ \delta \mathbf \Omega}\times ({\mathbf  \Omega}+\varphi({\mathbf \Omega}))
\]
In particular if 
\begin{equation}\label{eq:usual_RB_lagrangian}
l(\mathbf{\Omega})=\frac{1}{2}\mathbb{I}\mathbf{\Omega}\cdot \mathbf{\Omega}=\frac{1}{2}(I_1\Omega_1^2+I_2\Omega_2^2+I_3\Omega_3^2)
\end{equation}
where $\mathbb{I}=\operatorname{diag}(I_1,I_2,I_3)$, 
then equations (\ref{fel-1}) are
\begin{align*}%\label{relaxed}
	I_1\dot{\Omega}_1&=(I_2-I_3)\left(\Omega_2\Omega_3+\varphi_2({\mathbf \Omega})\Omega_3 -\varphi_3({\mathbf \Omega})\Omega_2\right)\notag\\
	I_2\dot{\Omega}_2&=(I_3-I_1)\left(\Omega_3\Omega_1+\varphi_3({\mathbf \Omega})\Omega_1-\varphi_1({\mathbf \Omega})\Omega_3\right)\notag\\
	I_3\dot{\Omega}_3&=(I_1-I_2)\left(\Omega_1\Omega_2+\varphi_1({\mathbf \Omega})\Omega_2-\varphi_2({\mathbf \Omega})\Omega_1\right)\,.\tag*{\ensuremath{\diamond}}
\end{align*}
}
\end{example}

The following proposition will be useful for the derivation of geometric integrators. 
\begin{proposition}\label{propo1}
Let $l\colon {\mathfrak g}\rightarrow {\mathbb R}$ be a Lagrangian function defined in the Lie algebra of a Lie group $G$ and $\varphi\colon {\mathfrak g}\rightarrow {\mathfrak g}$ be an arbitrary map. Then the following are equivalent: 
\begin{enumerate}	
	\item[a)] The forced Euler--Poincar\'e equations holds: 
	\[
	\frac{d}{dt}\left(\frac{\delta l}{ \delta \xi}\right)=ad_{\xi+\varphi(\xi)}^*	\frac{\delta l}{ \delta \xi}.
	\]
	\item[b)] The generalized variational principle
	\[
	{\boldsymbol \delta}\int_0^T l(\xi(t))\; dt=0,
	\]
	holds on ${\mathfrak g}$, using variations of the form: 
	\[
	{\boldsymbol \delta} \xi=\dot \eta +[\xi+\varphi(\xi), \eta], 
	\]
	for any $\eta\colon [0, T]\rightarrow {\mathfrak g}$ vanishing at the initial and end points: $\eta(0)=\eta(T)=0$. 
	\end{enumerate}
\end{proposition}
\begin{proof}
The equivalence follows from the direct computation
\begin{align*}
0&={\boldsymbol \delta} \int_0^T l(\xi(t))\; dt=\int_0^T \frac{\delta l}{\delta \xi}{\boldsymbol \delta}\xi\; dt\\ 
	&=\int_0^T \frac{\delta l}{\delta \xi}(\dot \eta +[\xi+\varphi(\xi), \eta])\, dt\\ 
	&=\int_0^T \left[ \frac{\delta l}{\delta \xi}[\xi+\varphi(\xi), \eta]-\frac{d}{dt}\left(\frac{\delta l}{\delta \xi}\right)\eta \right]\; dt\\
	&=\int_0^T \left[ ad^*_{\xi+\varphi(\xi)}\frac{\delta l}{\delta \xi}-\frac{d}{dt}\left(\frac{\delta l}{\delta \xi}\right)\right]\eta\; dt\; . \qedhere
	\end{align*}
\end{proof}

The Hamiltonian version of Equation (\ref{fel2-1}) is precisely:

\begin{equation}\label{flp-2}
\frac{d\mu}{dt}=ad^*_{\delta H/\delta\mu +\varphi(\delta H/\delta\mu)} \mu.
\end{equation}

\begin{proposition}\label{proposition-var}
If $\mu\colon [0, T]\rightarrow {\mathfrak g}^*$ is a solution of Equation \eqref{fel2-1} then along this solution
\begin{align*}
	\frac{dC}{dt}&=0 \quad \hbox{for any Casimir function},\\
	\frac{dH}{dt}&= \left\langle \mu, \left[ \varphi\left(\frac{\delta H}{\delta \mu}\right), \frac{\delta H}{\delta \mu}\right]\right\rangle.
\end{align*}
\end{proposition}

\begin{proof}
	Observe that 
	\begin{align*}
	\frac{dC}{dt}(t)&=\langle \dot{\mu}, \frac{\delta C}{\delta \mu}\rangle= 
    \langle ad^*_{\delta H/\delta\mu +\varphi(\delta H/\delta\mu)} \mu, \frac{\delta C}{\delta \mu}\rangle
    \\
	&= \left\langle \mu,  \left[ \frac{\delta H}{\delta \mu}+\varphi\left(\frac{\delta H}{\delta \mu}\right), \frac{\delta C}{\delta \mu}\right] \right\rangle=0,
		\end{align*}
	by the definition of the Casimir for the Lie--Poisson bracket.
	The other equality is proven in a similar way.
\end{proof}

\begin{example}{Rigid body. }{\rm 
	
	The rigid body equations in body coordinates are
	\[
	\dot{M}=M \times ({\mathbf\Omega}+\varphi({  \mathbf \Omega})),
	\]
	where $M=\mathbb{I}{\mathbf \Omega}=\frac{\partial l}{\partial \mathbf \Omega}$.}
\end{example}

We are now interested in constructing a particular family of mappings $\varphi\colon {\mathfrak g}\rightarrow {\mathfrak g}$ which reflect the {\bf dissipative} nature of a force; that is, we want to impose the following property:  
\[
\frac{dH}{dt}\leq 0.
\]
We will assume in the sequel that $G$ is finite dimensional since we are interested in designing numerical methods. 

Alternatively, the dissipation of energy expressed in the  Lagrangian formalism is
\begin{align*}
	\frac{dE_l}{dt}(\mu)&=\left\langle \frac{\partial l}{\partial \xi}, \left[ \varphi(\xi), \xi\right]\right\rangle\\
	&= \frac{\partial l}{\partial \xi^d}C_{ab}^d \varphi^a(\xi)\xi^b\leq 0
\end{align*}
where $ E_l(\xi)=\frac{\partial l}{\partial \xi}(\xi)\xi-l(\xi)$.

If $k\colon {\mathfrak g}^*\times {\mathfrak g}^*\rightarrow {\mathbb R}$ is a semi-definite positive scalar product and 
$k^\sharp\colon {\mathfrak g}^*\rightarrow {\mathfrak g}$ the induced musical isomorphism ($\langle \mu_2, k^\sharp(\mu_1)\rangle=k(\mu_1, \mu_2) $) then an obvious choice for $\varphi\colon {\mathfrak g}\rightarrow {\mathfrak g}$ is
\begin{equation}\label{varphi}
\varphi(\xi)=k^{\sharp}\left(ad^*_{\xi}\frac{\delta l}{\delta \xi}(\xi)\right)
\end{equation}
 and then
\[
\frac{dE_l}{dt}=-\left\langle ad_{\xi}^*\frac{\delta l}{\delta \xi}, \varphi(\xi)\right\rangle =-k(ad^*_{\xi}\frac{\delta l}{\delta \xi}(\xi), ad^*_{\xi}\frac{\delta l}{\delta \xi}(\xi))\leq 0
\]
In coordinates
\[
\varphi(\xi)= k^{ae}\frac{\partial l}{\partial \xi^d}C_{eb}^d \xi^b e_a
\]
where $\{e_a\}$ is a basis of the Lie algebra with dual basis $\{e^a\}$ of ${\mathfrak g}^*$ and $k(e^a, e^b)=k^{ab}$. 

\begin{example}{Forced rigid body (continued).   }
	{\rm 
In the case of the rigid body, considering the scalar product
\[
k ( (x_1, x_2, x_3), (y_1, y_2, y_3))=\alpha (x_1y_1+x_2y_2+x_3y_3)
\]
with $\alpha$ a positive constant, 
then 
\[
\varphi({\mathbf \Omega})=\alpha \mathbb{I}\mathbf{\Omega}\times\mathbf{\Omega}=(
\alpha(I_2-I_3)\Omega_2\Omega_3, 
\alpha(I_3-I_1)\Omega_3\Omega_1, 
\alpha(I_1-I_2)\Omega_1\Omega_2)
\]
and the equations of motion are
\begin{align*}%\label{relaxed-3}
	I_1\dot{\Omega}_1&=(I_2-I_3)\Omega_2\Omega_3+	\alpha(I_3-I_1)\Omega_1\Omega_3^2 -	\alpha(I_1-I_2)\Omega_1\Omega_2^2\\
	I_2\dot{\Omega}_2&=(I_3-I_1)\Omega_3\Omega_1+	\alpha(I_1-I_2)\Omega_2\Omega^2_1-\alpha(I_2-I_3)\Omega_2\Omega^2_3\\\
	I_3\dot{\Omega}_3&=(I_1-I_2)\Omega_1\Omega_2+	\alpha(I_2-I_3)\Omega_3\Omega^2_2-\alpha(I_3-I_1)\Omega_3\Omega^2_1\ .
\end{align*}

Observe that the equations are 
\begin{equation}\label{eq:diss_rbd}
\dot{{\mathbf M}}={\mathbf M}\times {\mathbf \Omega}+\alpha {\mathbf M}\times ({\mathbf M}\times {\mathbf \Omega}),
\end{equation}
where ${\mathbf M}=\mathbb{I}{\mathbf \Omega}$ (see  \citep{bloch1996euler} and  references therein). \hfill $\diamond$
}
\end{example}
Define on ${\mathfrak g}^*$ the bracket of functions
\[
\{\hspace{-0.1cm}\{f, g\}\hspace{-0.1cm}\}(\mu)=-\left\langle \mu, \left[ \varphi\left(\frac{\delta f}{\delta \mu}\right), \frac{\delta g}{\delta \mu}\right]\right\rangle,
\]
where $f, g \in C^{\infty}({\mathfrak g}^*)$.
In the case of the map $\varphi$ specified in \eqref{varphi} we have that
\[
\{\hspace{-0.1cm}\{f, g\}\hspace{-0.1cm}\}(\mu)=-\left\langle \mu, \left[ k^\sharp (ad^*_{\delta f/\delta \mu}\mu), \frac{\delta g}{\delta \mu}\right] \right \rangle,
\]
is a symmetric bracket, that is $\{\hspace{-0.1cm}\{f, g\}\hspace{-0.1cm}\}=\{\hspace{-0.1cm}\{g, f\}\hspace{-0.1cm}\}$.

Then the equations of motion in double-bracket form are: 
\[
\dot{f}=\{f, H\}-\{\hspace{-0.1cm}\{f, H\}\hspace{-0.1cm}\},
\]
or in matrix form
\[
\dot{z}=\Pi(z) \nabla H(z)- S(z)\nabla H (z),
\]
where $S=\Pi^T K\Pi=(S_{ad})=(\Pi_{ba}k^{bc}\Pi_{cd})$, where $K$ is the matrix corresponding to the linear map $k^{\sharp}$.

We know by construction that
\[
S\nabla C=0,  \hbox{ for all Casimirs } C\colon {\mathfrak g}^*\rightarrow {\mathbb R}. 
\]

\section{Discrete Lagrangian formalism for double-bracket dissipation}\label{sec:Discrete_Lagrangian_formalism_for_DBD}

%(See for more details \cite{weinstein96,MMM06Grupoides}).

In this section, we propose a discretization of the  equations with double-bracket dissipation  based on Proposition \ref{proposition-var}  (``DDB methods'').

As a set of discrete curves we consider sequences of elements of the Lie group $G$, that is, $\{w_k\}_{k=1, \ldots, N}$ (see \cite{MaPeSh99} for details).

Fixing $w_{f}\in G$, we define the set of admissible pairs
\[
C_{w_f}^2=\{(w_1, w_2)\in G\times G\; |\; w_1w_2=w_f\}\; .
\]
A tangent vector to the manifold $C_{w_f}^2$ is a tangent vector at $t=0$ of a curve in $C_{w_f}^2$
\[
t\in (-\epsilon, \epsilon)\subseteq \R\longrightarrow (c_1(t), c_2(t))
\]
where $c_i(t)\in G$, $c_1(t)c_2(t)=w_f$ and $c_1(0)=w_1$ and $c_2(0)=w_2$. 
All such curves are given by
\begin{equation}\label{curves}
c(t)=(w_1\gamma(t), \gamma(t)^{-1}w_2)
\end{equation}
for an arbitrary $\gamma(t)\in G$ with $t\in (-\epsilon, \epsilon)$ and $\gamma(0)=e$, where $e$ is the neutral element of $G$. Observe that if $(c_1(t), c_2(t))$ is a curve in $C_{w_f}^2$ then if we define $\gamma(t)=w_1^{-1}c_1(t)$, then 
\begin{align*}
w_1\gamma(t)&=c_1(t)\\
\gamma(t)^{-1}w_2&=(c_1(t))^{-1}w_1w_2=(c_1(t))^{-1}w_f=c_2(t)
\end{align*}

To generate the discrete dynamics, as in the calculus of variations, we prescribe a Lagrangian function, called a discrete Lagrangian defined as $l_d\colon G\rightarrow {\mathbb R}$. 
Given $l_d\colon G\rightarrow \R$, we define the {\bf discrete action sum} by
\[
\begin{array}{rrcl}
S_{l_d}\colon&C^2_w&\longrightarrow&\R\\
&(w_1, w_2)&\longmapsto& l_d(w_1)+l_d(w_2)
\end{array}
\]

\begin{definition}{\bf Discrete Hamilton's principle}.
	Given $w_f\in G$, then $(w_1, w_2)\in C_{w_f}^2$ is a solution of the discrete Lagrangian system determined by $l_d\colon G\rightarrow \R$ if and only if $(w_1, w_2)$ is a critical point of $S_{l_d}$.
\end{definition}

We characterize the critical points  using the curves defined in (\ref{curves}) as follows
\begin{align*}
	0&= \frac{d}{dt}\Big|_{t=0}S_{l_d}(c(t))\\
	&= \frac{d}{dt}\Big|_{t=0}\left( l_d(w_1\gamma(t))+l_d(\gamma(t)^{-1}w_2)\right)\\
	&= d(l_d\circ {\mathcal L}_{w_1})(e)(\xi)- d(l_d\circ {\mathcal R}_{w_2})(e)(\xi)
\end{align*}
where $\xi=\dot{h}(0)$. 

Alternatively, we can write these equations as 
\begin{equation}\label{dep}
0=\lvec{\xi}(w_1)(l_d)-\rvec{\xi}(w_2)(l_d)\; , \quad \forall \xi \in {\mathfrak g}
\end{equation}
which are called {\bf discrete Euler--Poincar\'e equations}. Here $\lvec{\xi}(w)=T_e{\mathcal L}_{w}(\xi)$ and   $\rvec{\xi}(w)=T_e{\mathcal R}_{w}(\xi)$ are the left- and right-invariant vector fields corresponding to $\xi$, respectively. 

Also it is possible to define the two discrete Legendre transformations  (see \cite{MMM}) by:  
$Fl_d^-\colon  G\rightarrow {\mathfrak g}^*$ and $Fl_d^+\colon  G\rightarrow {\mathfrak g}^*$ by
\begin{align*}
	Fl_d^-(w)&={\mathcal L}_w^* dl_d(w)\\
	Fl_d^+(w)&={\mathcal R}_w^* dl_d(w).
\end{align*}
So, if we define
\begin{equation}\label{eq:def_of_mu}
M_k=Fl_d^+(w_k)={\mathcal R}_{w_k}^* dl_d(w_k),
\end{equation}
then Equations \eqref{dep} are equivalent to
\[
M_{k+1}=Fl_d^+(w_{k+1})=Fl_d^-(w_k)=Ad^*_{w_k}M_k,
\]
which in this case are called {\bf discrete Lie--Poisson equations}. Then, it defines an implicit map $M_k \mapsto M_{k+1}$ preserving the Lie--Poisson structure and, in particular, the coadjoint orbits. 

From the previous discussion, it is interesting to observe that the discrete Euler--Poincar\'e equations are derived as follows. 
\begin{proposition}\label{propo1discrete}
Let $l_d\colon G\rightarrow {\mathbb R}$ be a discrete Lagrangian function defined in a Lie group $G$. Then the following are equivalent: 
\begin{enumerate}	
	\item[a)] The discrete Euler--Poincar\'e equations hold: 
	\[
0=\lvec{\xi}(w_1)(l_d)-\rvec{\xi}(w_2)(l_d)\; , \quad \forall \xi \in {\mathfrak g}.
\]
	\item[b)] The generalized variational principle
	\[
	{\boldsymbol \delta}\sum_{k=0}^{N-1}l_d (w_k)\; =0,
	\]
 holds for variations of the form: 
	\[
	{\boldsymbol \delta} w_k=\eta_k w_k+w_k\eta_{k+1}, 
	\]
	for any $\{\eta_k\}_{k=0}^{N}$
    with $\eta_k\in {\mathfrak g}$ vanishing at the initial and end points: $\eta_0=\eta_N=0$. 
	\end{enumerate}
\end{proposition}

\subsection{Forced Lie--Poisson integrators preserving the coadjoint orbits}

Motivated by Proposition \ref{propo1discrete}, we describe the following discrete generalized variational principle:  
\begin{theorem}\label{propo2}
	Let $l_d\colon G\rightarrow {\mathbb R}$ be a Lagrangian function defined on a Lie group $G$ and $\varphi_d\colon G\rightarrow G$ an arbitrary map. Then the following conditions are equivalent: 
	\begin{enumerate}	
		\item[a)] The sequence $\{w_k\}_{k=0}^{N-1}$ of elements of $G$ satisfies the  discrete forced Euler--Poincar\'e equations
%	
%	0=\lvec{\xi+m_d(\xi)}(g_k)(l_d)-\rvec{\xi+m_d(\xi)}(g_{k+1})(l_d)\; , \quad \forall \xi \in {\mathfrak g}
%	\end{equation}
%	or, alternatively, 
\begin{equation}\label{dep-1}
	0={\mathcal R}^*_{w_{k+1}} dl_d(w_{k+1})-Ad^*_{\varphi_d(w_{k})}{\mathcal L}^*_{w_k}dl_d(w_k),
\end{equation}
	for $k=0, \dots, N-2$. 
		\item[b)] The discrete generalized variational principle
		\[
		{\boldsymbol \delta} \sum_{k=0}^{N-1} l_d(w_k)=0,
		\]
		holds for variations of the form: 
		\begin{equation}\label{wer}
		{\boldsymbol \delta} w_k=-\eta_k w_k +  w_k (Ad_{\varphi_d(w_{k})}\eta_{k+1}),
		\end{equation}
		for any $\{\eta_k\}_{k=0}^N$ with $\eta_0=\eta_N=0$. 
	\end{enumerate}
\end{theorem}
\begin{proof}
	
	Computing the variation of the discrete action
	\begin{align*}
0&={\boldsymbol \delta} \sum_{k=0}^{N-1} l_d(w_k)=\sum_{k=0}^{N-1} \langle dl_d(w_k),{\boldsymbol \delta} w_k\rangle \\
 &=\sum_{k=0}^{N-1} \langle dl_d(w_k), -\eta_k w_k +  w_k (Ad_{\varphi_d(w_{k})}\eta_{k+1})\rangle\\
 &=-\sum_{k=0}^{N-2}\left\langle {\mathcal R}^*_{w_{k+1}}dl_d(w_{k+1}), \eta_{k+1}\right\rangle 
+\sum_{k=0}^{N-2}\left\langle Ad_{\varphi_d(w_{k})}^* {\mathcal L}^*_{w_{k}} dl_d(w_k), \eta_{k+1}\right\rangle 
	\end{align*}
we obtain Equation~\eqref{dep-1}. 	
\end{proof}

Defining, as in Equation~\eqref{eq:def_of_mu}, the \textbf{discrete momentum}  
\begin{equation}\label{eq:discrete_momentum_Mk_with_R}
	M_k={\mathcal R}_{w_k}^*dl_d(w_k),
\end{equation}
we can write Equation~\eqref{dep-1} as $M_{k+1}=Ad^*_{\varphi_d(w_{k})}Ad^*_{w_k}M_k$, or equivalently as the \textbf{forced discrete Lie--Poisson equations}
\begin{equation}\label{eq:evolution_of_Mk}
	M_{k+1}=Ad^*_{w_k\varphi_d(w_{k})}M_k.
\end{equation}
From this we can see that the discrete flow $M_k\mapsto M_{k+1}$ preserves the coadjoint orbits. 

\begin{remark}
It is possible to understand equation 
\[
{\boldsymbol \delta} w_k=-\eta_k w_k +  w_k (Ad_{\varphi_d(w_{k})}\eta_{k+1}),
\]
as an approximation to the continuous equation: 
\[
	\delta \xi=\dot \eta +[\xi+\varphi(\xi), \eta], 
	\]
In fact, if we define
\[
w_k=\hbox{exp} (h\xi_k),
\]
where $\xi_k\in {\mathfrak g}$ then
\begin{align*}
\delta w_k&= T_{h\xi_k}\hbox{exp}(h\delta \xi_k)\\
&=T_e{\mathcal R}_{w_k}(\hbox{dexp}_{h\xi_k}(h\delta\xi_k))=\hbox{dexp}_{h\xi_k}(h\delta\xi_k)w_k\\
&=\left(\sum_{j=0}^{\infty}
\frac{1}{(j+1)!} \hbox{ad}^j_{h\xi_k}(h\delta \xi_k)\right) w_k\\
&=h\delta\xi_k w_k+O(h^2)\\
&=(\eta_{k+1}-\eta_k)w_k+h\; (\hbox{ad}_{\xi_k+\varphi(\xi_k)}\eta_{k+1})w_k +O(h^2)\\
&=(\eta_{k+1}-\eta_k)w_k
+\left(\hbox{Ad}_{\hbox{\small exp}
(h\xi_k+h\varphi(\xi_k))} \eta_{k+1}-\eta_{k+1}\right)w_k+O(h^2)\\
&=-\eta_k w_k+w_k \hbox{Ad}_{\hbox{\small exp}
(h\varphi(\xi_k))} \eta_{k+1}+O(h^2)\\
&=-\eta_k w_k+w_k \hbox{Ad}_{\varphi_d(w_k)} \eta_{k+1}+O(h^2),
\end{align*}
where we have defined  $\varphi_d(w_k)=\hbox{exp}
(h\varphi(\xi_k))$.
In the previous calculations we have made use of the right trivialized tangent map of the exponential map defined by (see Definition \ref{Retr} for more details):
\[
T_{\xi}\hbox{exp}=T_eR_{\hbox{\small exp}(\xi)}\circ \hbox{dexp}_{\xi},
\]
and that \citep{Iserles_Munthe-Kaas_Nørsett_Zanna_2000}
\[
\hbox{dexp}_{\xi}=\sum_{j=0}^{\infty}\frac{1}{(j+1)!}ad^j_{\xi}.
\]
\end{remark}

	\begin{example}{\bf Discrete forced Rigid Body Equations.}\label{example_forced_RB_MV} (\cite{moser-veselov0138}, see also e.g. \cite{bloch2002symmetric})
		{\rm 
			It is possible to write the Lagrangian \eqref{eq:usual_RB_lagrangian} as
			\begin{equation}\label{eq:arnold}
l(\Omega)=\frac{1}{2}\operatorname{trace}( \Omega J \Omega^T),
			\end{equation}
            where $\Omega$ is the skew-symmetric matrix associated to $\mathbf\Omega$, usually denoted also by $\widehat{\mathbf{\Omega}}=\Omega$. The reconstruction equation is $\dot{R}=R\Omega$ where $R \in SO(3)$.
            The matrix $J$ is related to $\mathbb{I}$ by
            \[ J=\frac{1}{2}\operatorname{trace}(\mathbb{I})I-\mathbb{I}=\frac{1}{2}\operatorname{diag}(-I_1+I_2+I_3,I_1-I_2+I_3,I_1+I_2-I_3).\]
            For $\mathbb{I}$ to be physically valid, the principal moments of inertia must satisfy the triangle inequality, which is equivalent to $J$ being a nonnegative matrix. Nevertheless, our derivations below do not rely on this assumption.
            
            Although we defined $l$ for a skew-symmetric $\Omega$,
            for approximation purposes it is possible to evaluate it on arbitrary matrices. Indeed, a discretization of this Lagrangian is given by the discrete Lagrangian $l_d\colon SO(3)\rightarrow {\mathbb R}$: 
			\[
			l_d(w_k)=hl\left( \frac{w_k-I}{h}\right)=\frac{h}{2}\operatorname{trace}\left(  \left( \frac{w_k-I}{h}\right)J\left(\frac{w_k-I}{h}\right)^T\right),
			\]
			where we have resorted to the approximation
			\[\Omega_k=R_k^{-1}\dot{R}_k\approx R_k^{-1}\left( \frac{R_{k+1}-R_k}{h}\right)=\frac{R_k^{-1} R_{k+1}-I}{h}= \frac{w_k-I}{h},\]
			with $w_k=R_k^{-1}R_{k+1}=R_k^TR_{k+1}$.
			
			A straightforward computation shows that $l_d(w_k)=\left(\operatorname{trace}(J)-\operatorname{trace}(w_k J)\right)/h$.
			Applying the characterization of critical points discussed at the beginning of Section~\ref{sec:Discrete_Lagrangian_formalism_for_DBD}, we have that
			\begin{align*}
				0&=\frac{d}{dt}\Big|_{t=0}\operatorname{trace} (w_k\gamma(t)J)+\operatorname{trace} (\gamma(t)^{-1}w_{k+1}J)\\
				&=\operatorname{trace} (\dot{\gamma}(0)(Jw_k+Jw_{k+1}^T)),
			\end{align*}
			for all $\gamma\colon (-\epsilon,\epsilon)\rightarrow SO(3)$ with $\gamma(0)=I$. Since $\dot \gamma(0)\in\mathfrak{so}(3)$ is arbitrary, this is equivalent to $Jw_k+Jw_{k+1}^T$ being symmetric. Then the discrete Euler--Poincar\'e equations are
			\[
			Jw_k+Jw_{k+1}^T=w_k^TJ+w_{k+1}J.
			\]
			
			Also, it is possible to write the rigid body equations as discrete Lie--Poisson equations (see \cite{moser-veselov0138}):
			\[
			M_{k+1}=Ad_{w_k}^*M_k=w_k^T M_k w_k,
			\]
			where $M_k={\mathcal R}_{w_k}^*dl_d(w_k)=\frac{1}{h}(w_k J-Jw_k^T)$ is a skew-symmetric matrix.
			
	Now, we add in this picture the discrete (dissipative) map $\varphi_d\colon SO(3)\rightarrow SO(3)$ then the discrete Lie--Poisson equations are replaced by their forced version~\eqref{eq:evolution_of_Mk}, that is, $	M_{k+1}=Ad^*_{w_k\varphi_d(w_{k})}M_k$. Let us solve these equations on the interval $[t_0, t_N]$, where $N\in \mathbb{N}$, $h=(t_N-t_0)/N$ is the step size of integration and $t_k = t_0 + kh$. The initial angular momentum is $\mu(t_0)\in\mathfrak{so}(3)^*$, where $\mu(t)=\widehat{\mathbb{I}\mathbf{\Omega}}=\Omega(t)J +J\Omega(t)$ and $\Omega(t)$ is a solution of the double-bracket dissipation equations for  the continuous Lagrangian given in (\ref{eq:arnold}) and dissipation given by $\varphi$. These are related by $\varphi_d(w_k)=\exp(h\varphi(\Omega_k))$, but we can take other retractions, such as the Cayley map, as discussed in the next section.
    Note that since $w_k$ represents an increment in the group variable rather than an angular velocity, the continuous angular momentum $\mu$ at time $t_k=t_0+kh$ is approximated by
    \begin{align*}
    M_k&=\frac{1}{h}(w_k J-Jw_k^T) \\
    &\approx \frac{1}{h}((I+h\Omega_k)J-J(I+h\Omega_k)^T\\
    &=\Omega_kJ+J\Omega_k=\mu(t_k)\,.
    \end{align*}

Then the method for solving the discrete Lie--Poisson equations for the rigid body with double-bracket dissipation, in matrix formulation, is the following:
    \begin{enumerate}
			\item Set $M_0=\mu(t_0).$
			\item For $k=0,\ldots, N-1$, 
            \begin{enumerate}
                \item find $w_k$ such that $M_k=\frac{1}{h}(w_k J-Jw_k^T)$;
            \item set
			$M_{k+1}=Ad^*_{w_k\varphi_d(w_{k})}M_k=(w_k\varphi_d(w_{k}))^T M_k w_k\varphi_d(w_{k})$;
			\item output the approximation $\mu(t_{k+1})\approx M_{k+1}$.
            \end{enumerate}
		\end{enumerate}

        It is worth mentioning that solving for $w_k$ in each step involves selecting a branch in a correspondence (see~\cite{moser-veselov0138}, \S 1.2), which agrees with the fact that discrete equations of motion, such as the discrete Euler--Lagrange equation, can generally have multiple solutions.
        
		}
	\end{example}
	\subsection{Discrete methods with double-bracket dissipation using retraction maps}
For constructing the method starting from the continuous data $l$ and $\varphi$, we will need to have a correspondence between
the Lie groups variables $w_k$ and the averaged velocity of the
	displacement they represent. This is accomplished using a {\bf retraction map} $\tau\colon {\mathfrak g}\to G$ which is an analytic local diffeomorphism mapping a neighborhood of $0\in\mathfrak{g}$ to a neighborhood of the identity $e\in G$, such that $\tau(0)=e$ and $\tau(\xi)\tau(-\xi)=e$, where $\xi\in\mathfrak g$. 
	
	 We define  the discrete Lagrangian $l_d\colon G\to
{\mathbb R}$ to be
	\begin{equation*}
		l_d(w_k)=h\,l(\xi_k),
	\end{equation*}
	where $\xi_k=\tau^{-1}(w_k)/h$. 
	The remaining ingredient is to define the map $\varphi_d\colon G\rightarrow G$ from $\varphi\colon {\mathfrak g}\rightarrow {\mathfrak g}$ as follows:
	\begin{equation}\label{eq:def_of_phid}
\varphi_d(w_k)=\tau(h \varphi(\xi_k)).
	\end{equation}
	This is actually defined for $w_k$ in a neighborhood of $e\in G$.
    
	Now, we can express the discrete equations for double-bracket dissipation in terms of a chosen map $\tau$. The resulting discrete mechanics will thus involve the derivatives of the map which we define next (see also~\cite{Bou-Rabee2009,Iserles_Munthe-Kaas_Nørsett_Zanna_2000}):
	\begin{definition}\label{Retr}
		{\rm Given a map $\tau\colon\mathfrak{g}\to G$, its {\bf right trivialized tangent} $\mbox{d}\tau_{\xi}\colon\mathfrak{g}\to\mathfrak{g}$ and its {\bf inverse} $\mbox{d}\tau_{\xi}^{-1}\colon\mathfrak{g}\to\mathfrak{g}$, are $\xi$-dependent maps, $\xi\in\mathfrak{g}$, such that for  $\eta\in\mathfrak{g}$, the following holds}
		\begin{align*}
			&T_{\xi}\tau (\eta)=\mbox{d}\tau_{\xi} (\eta)\,\tau(\xi),\\
			&T_{\tau(\xi)}\tau^{-1}(\eta\tau(\xi))=\mbox{d}\tau^{-1}_{\xi}(\eta).
		\end{align*}
	\end{definition}
	Using these definitions, from the variations given by 
    \[
{\boldsymbol \delta} w_k=-\eta_k w_k +  w_k (Ad_{\varphi_d(w_{k})}\eta_{k+1}),
\]
 we can rewrite this expression as variations in terms of $\delta \xi_k$  
    as follows: 
    \begin{align*}
        \delta{\xi}_k&=\frac{1}{h}T_{w_k}\tau^{-1}(\delta w_k)\\
        &=\frac{1}{h}d\tau^{-1}_{h\xi_k}(-\eta_k + Ad_{\tau(h\xi_k)\varphi_d(w_{k})}\eta_{k+1}),
    \end{align*}
     where we have used differentiation of $\xi_k=\tau^{-1}(w_{k})/h$ and  expression (\ref{wer}).
	
	Now, with these discretizations of the continuous Lagrangian and the map $\varphi\colon {\mathfrak g}\rightarrow {\mathfrak g}$, the  discrete generalized variational principle derived in Theorem \ref{propo2} is now rewritten as 
	\[
	\delta \sum_{k=0}^{N-1} l(\xi_k)=0,
	\]
	for variations of the form: 
\begin{align*}
\delta \xi_k&=\frac{1}{h}\hbox{d}\tau^{-1}_{h\xi_k}(-\eta_k+\hbox{Ad}_{\tau(h\xi_k)\tau(h \varphi(\xi_k))}
\eta_{k+1}),\\
\\
&=-\frac{1}{h}\hbox{d}\tau^{-1}_{h\xi_k}(\eta_k)+\frac{1}{h}\hbox{d}\tau^{-1}_{-h\xi_k}(\hbox{Ad}_{\tau(h \varphi(\xi_k))}
\eta_{k+1}),\
\end{align*}
for any $\{\eta_k\}_{k=0}^N$ with $\eta_0=\eta_N=0$.
	
We derive the  following discrete equations of motion for systems with double-bracket dissipation: 
\begin{equation}\label{eq-doubledissipation-lagrangian}
-\left(\hbox{d}\tau^{-1}_{h\xi_{k+1}}\right)^* \frac{\partial l}{\partial \xi}(\xi_{k+1})
+Ad_{\tau(h \varphi(\xi_k))}^*
\left(\hbox{d}\tau^{-1}_{-h\xi_{k}}\right)^* \frac{\partial l}{\partial \xi}(\xi_{k})=0,
\end{equation}
where $\left(\hbox{d} \tau^{-1}_{\xi} \right)^* \colon \mathfrak{g}^* \to \mathfrak{g}^*$ for each $\xi\in\mathfrak{g}$. We can denote
	\[
	p_{k}=\frac{\partial l}{\partial \xi}(\xi_{k}),
	\]
	so the equations become 
	\begin{equation}\label{eq:evolution_of_pk}
(\hbox{d}\tau^{-1}_{h\xi_{k+1}}) ^*	p_{k+1}=Ad_{\tau(h \varphi(\xi_k))}^*
(\hbox{d}\tau^{-1}_{-h\xi_{k}}) ^*p_k.
	\end{equation}
By using the relation
\[
\left( \hbox{d} \tau^{-1}_{-h\xi_k} \right)^* = Ad_{\tau(h\xi_k)}^* \left( \hbox{d} \tau^{-1}_{h\xi_k} \right)^*,
\]
we get
\[
\left( \hbox{d} \tau_{h\xi_{k+1}}^{-1} \right)^* p_{k+1} = Ad_{\tau(h \varphi(\xi_k))}^* Ad_{\tau(h \xi_k)}^* \left(\hbox{d} \tau^{-1}_{h\xi_k} \right)^* p_k.
\]

Call $M_k = \left(\hbox{d} \tau_{h \xi_k}^{-1} \right)^* p_k$. This is consistent with our previous definition~\eqref{eq:discrete_momentum_Mk_with_R} of the discrete momentum. Indeed, $l_d(w_k)=hl(\xi_k)=hl(\tau^{-1}(w_k)/h)$, so
\[
\mathcal{R}^*_{w_k}dl_d(w_k)=
\mathcal{R}^*_{w_k}\left((T_{h\xi_k}\tau^{-1})^*dl(\xi_k)\right)=
(T_{h\xi_k}\tau^{-1}\circ \mathcal{R}_{w_k})^* dl(\xi_k)=\left(\hbox{d} \tau_{h \xi_k}^{-1} \right)^* p_k.
\]
Thus, we obtain $
M_{k+1} = Ad^*_{\tau(h\xi_k) \tau(h \varphi(\xi_k)) } M_k$, or equivalently,
\begin{equation}\label{eq:evolution_of_Mk_in_tau_section}
M_{k+1} = Ad^*_{w_k \varphi_d(w_k)} M_k.
\end{equation}
\begin{remark}
Equation~\eqref{eq:evolution_of_Mk_in_tau_section} coincides with~\eqref{eq:evolution_of_Mk}. In fact, the matrix $w_k$ in~\eqref{eq:evolution_of_Mk} is obtained by solving~\eqref{eq:discrete_momentum_Mk_with_R}, which is equivalent to solving $M_k=\left(\hbox{d} \tau_{h \xi_k}^{-1} \right)^* \frac{\partial l}{\partial \xi}(\xi_k)$ for $\xi_k$.
\end{remark}

\begin{remark}\label{rem:Mk_vs_pk}
    Equation~\eqref{eq:evolution_of_Mk_in_tau_section} shows that the method, as an update rule $M_k\mapsto M_{k+1}$, preserves the coadjoint orbits. The same cannot be said of the evolution of $p_k$. These can be interpreted as a discretization of the continuous momentum $\frac{\partial l}{\partial \xi}$ as opposed to the discrete momentum $M_k$, which is defined from $l_d$. If $T_0\tau=d\tau_0$ is the identity map, a condition sometimes included in the definition of a retraction map, then both $M_k$ and $p_k$ approximate the continuous momentum at each time step. However, only $M_k$ is guaranteed to evolve on its coadjoint orbit.
\end{remark}

		Therefore, we obtain the following discrete  method with double-bracket dissipation 
({\bf DDB method}) in terms of a retraction map $\tau$:

\begin{enumerate}
			\item Set $M_0=\mu(t_0).$
			\item For $k=0,\ldots, N-1$, 
            \begin{enumerate}
                \item find $\xi_k$ such that $M_k = \left(\hbox{d} \tau_{h \xi_k}^{-1} \right)^* \displaystyle\frac{\partial l}{\partial \xi}(\xi_k)$;
            \item set $w_k=\tau(h\xi_k)$
            \item set $M_{k+1}=Ad^*_{w_k\varphi_d(w_{k})}M_k$;
			\item output the approximation $\mu(t_{k+1})\approx M_{k+1}$.
            \end{enumerate}
		\end{enumerate}

For step (ii)(c), one can avoid computing the expression of $\varphi_d$ explicitly by using its definition~\eqref{eq:def_of_phid}; that is, one can set $M_{k+1}=Ad^*_{w_k \tau(h\varphi(\xi_{k}))}M_k$.

	Some typical choices for the retraction map $\tau$ are the following.
	\vspace{0.2cm}
	
	a) The exponential map $\e\colon\al\to G$, defined by $\exp(\xi)=\gamma(1)$, with $\gamma\colon\R\to G$ being the integral curve through the identity of the vector field associated with $\xi\in\al$ (hence, with $\dot\gamma(0)=\xi$). The right trivialized derivative and its inverse are defined by
	\begin{align}
		\mbox{d}\e_{x}\,y&=\sum_{j=0}^{\infty}\frac{1}{(j+1)!}\,\ad_{x}^{j}\, y,\label{eq:dexpxy}\\
		\mbox{d}\e_{x}^{-1}\,y&=\sum_{j=0}^{\infty}\frac{B_{j}}{j!}\,\ad_{x}^{j}\, y,\label{eq:dexpinvxy}
	\end{align}
	where $B_{j}$ are the Bernoulli numbers (see \cite{Iserles_Munthe-Kaas_Nørsett_Zanna_2000}). Typically, these expressions are truncated in order to achieve a desired order of accuracy.
	\vspace{0.2cm}
	
	b) The Cayley map $\ca\colon\al\to G$ is defined by $\ca(\xi)=(I-\frac{\xi}{2})^{-1}(I+\frac{\xi}{2})$ and is valid for a general class of quadratic groups (see \cite{Iserles_Munthe-Kaas_Nørsett_Zanna_2000}). It is applicable to groups of interest in many examples such as $SO(3)$, $SE(2)$ and $SE(3)$ (even though the special Euclidean groups are not quadratic).
    Its right trivialized derivative and inverse are
	\begin{align}
		\mbox{d}\ca_{\xi}\,\eta&=\left(I-\frac{\xi}{2}\right)^{-1}\eta\left(I+\frac{\xi}{2}\right)^{-1}\label{eq:dcay}\\
		\mbox{d}\ca_{\xi}^{-1}\eta&=\left(I-\frac{\xi}{2}\right)\eta\left(I+\frac{\xi}{2}\right).\label{eq:dcayinv}
	\end{align}

In particular, for the case of the group $SO(3)$ used in our rigid body example, let us consider the usual hat map $\widehat{\mbox{}}\,\colon \R^3\to \mathfrak{so}(3)$, $x\mapsto \widehat x=\begin{psmallmatrix}0&-x_3&x_2\\
x_3&0&-x_1\\
-x_2&x_1&0
\end{psmallmatrix}$.
Using this isomorphism we can rewrite equations \eqref{eq:dcay} and \eqref{eq:dcayinv} as linear operators on $\R^3$ with matrix expressions
\begin{align*}
		\mbox{d}\ca_{\widehat{x}}&=\frac{4}{4+\|x\|^2}\left(I+\frac{\widehat{x}}{2}\right)\\
		\mbox{d}\ca_{\widehat{x}}^{-1}&=I-\frac{\widehat{x}}{2}+\frac{x\otimes x}{4}.
	\end{align*}
These are essentially the expressions in \cite{CamposMartinTorrente2025preprint}, except for some constant factors arising from having taken a slightly different choice for the Cayley transform.

Since $\mathfrak{so}(3)^*\equiv \mathfrak{so}(3)\equiv \R^3$ with pairing $\langle\widehat \mu,\widehat \xi\rangle=\frac{1}{2}\tr(\widehat\mu\widehat\xi^T)=\mu\cdot \xi
$, then the duals of these operators are obtained by transposing. In particular,
$     \left(\mbox{d}\ca_{\widehat{x}}^{-1}\right)^*=I+\frac{\widehat{x}}{2}+\frac{x\otimes x}{4}$.

These ideas can also be applied to the case of the exponential map in $\mathfrak{so}(3)$. Indeed, $\ad_x y=x\times y =\widehat{x} y$, so~\eqref{eq:dexpxy} and~\eqref{eq:dexpinvxy} become
\begin{align*}
\operatorname{dexp}_{\widehat{x}}&=\sum_{j=0}^{\infty} \frac{\widehat{x}^j}{(j+1)!}\\
\operatorname{dexp}_{\widehat{x}}^{-1}&=\sum_{j=0}^{\infty} \frac{B_j\widehat{x}^j}{j!}.
\end{align*}
In particular, $\left(\operatorname{dexp}_{\widehat{x}}^{-1}\right)^*=\sum_{j=0}^{\infty} {(-1)^j B_j\widehat{x}^j}/{j!}$. It should be noted that $B_j=0$ for odd $j$, except for $B_1$ which is $-1/2$ in the convention we adopt for these formulas.

\begin{remark} We can rewrite \eqref{eq:evolution_of_pk} using \eqref{eq:dcayinv} as
	\[
	p_{k+1}-\frac{h}{2}\operatorname{ad}^*_{\xi_{k+1}}p_{k+1}-\frac{h^2}{4}\xi_{k+1}^Tp_{k+1}\xi^T_{k+1}
	=	\operatorname{Ad}^*_{\operatorname{cay}(h\varphi(\xi_k))}\left({p}_{k}+\frac{h}{2}\operatorname{ad}^*_{\xi_{k}}{p}_{k}-\frac{h^2}{4}\xi_{k}^T{p}_{k}\xi^T_{k}\right).
	\]
	For the case with no dissipation ($\varphi=0$), this recovers equation (4.26) in \cite{Bou-Rabee2009}. Recall, however, that the map $p_k\mapsto p_{k+1}$ does not preserve coadjoint orbits (Remark~\ref{rem:Mk_vs_pk}). 
\end{remark}

\begin{remark}
We can use the retraction maps to derive a second order discretization of the map $\varphi\colon {\mathfrak g}\rightarrow {\mathfrak g}$. For instance, we can take  
\[
\varphi_d(w_k)=\tau\left(h \varphi\left(\frac{\xi_k+\xi_{k+1}}{2}\right)\right),
\]
or 
\[
\varphi_d(w_k)=\tau\left(\frac{h}{2} \varphi(\xi_k)+\frac{h}{2} \varphi(\xi_{k+1})\right)\; .
\]
With this discrete map $\varphi_d$ the  equation (\ref{eq-doubledissipation-lagrangian}) is  unaltered after exchanging $\xi_k\leftrightarrow \xi_{k+1}$ and $h\leftrightarrow -h$.  
Then 
the method is symmetric and, as consequence, its order is two.  
\end{remark}

\section{Asymptotic behavior of the rigid body and numerical simulations}\label{sec5}
Continuing with our rigid body example, we know from Equation~\eqref{eq:diss_rbd} that the body momentum $M(t)$ evolves on the constant-Casimir sphere $\mathcal{S}= \{M \,|\, \|M\|=\|M(0)\|\}$, which is compact and (positively) invariant under the ODE. In addition, we have chosen $\varphi$ using Equation~\eqref{varphi}, which implies that $dE_l/dt = -\alpha \|M\times \Omega\|^2\leq 0$ on trajectories. LaSalle's invariance principle \citep[Thm.~1]{LaSalle1960} states that, under these conditions, all trajectories approach the largest invariant subset $\mathcal{M}$ of
\[\mathcal{E}=\{M\in \mathcal{S}\,|\,dE_l/dt = -\alpha \|M\times \Omega\|^2= 0\}=\{M\in \mathcal{S}\,|\,M\times \Omega=0\}.\]
This means that $M$ must be parallel to $\Omega=\mathbb{I}^{-1}M$. That is,  $\mathcal{E}$ is the set of eigenvectors of $\mathbb{I}$ with norm $\|M(0)\|$. For every $M\in\mathcal{E}$, the ODE~\eqref{eq:diss_rbd} has an equilibrium point there ($\dot M=0$), so $\mathcal{M}=\mathcal{E}$.

Consider the generic case where one of the moments of inertia is strictly larger than the other two. The corresponding principal axis determines two opposite points on $\mathcal{S}$ for which the energy is a minimum. For almost every initial condition on $\mathcal{S}$, $M(t)$ converges to either one of these points, that is, the body will end up rotating about its shortest axis. The remaining principal axes of inertia correspond to unstable equilibria.

If, alternatively, there are two equal moments of inertia that are strictly larger than the remaining one, the corresponding eigenspace has dimension 2 and $\mathcal{M}$ is a great circle. Since $\mathcal{M}$ consists of equilibrium points rather than merely being invariant, there is no motion tangent to it, which implies that each trajectory converges to some point on $\mathcal{M}$. This is the case, for example, for a long cylinder, and the eigenspace is a plane perpendicular to the axis of the cylinder. This situation and the previous one are illustrated in Figure~\ref{fig:spheresBC} below.

Finally, if the three moments of inertia are equal, then $\mathcal{M}=\mathcal{S}$ and the energy remains constant for every initial condition.

\paragraph{Numerical error analysis} Taking into account the discussion above, we compare the results of our discrete double-bracket dissipation method (with both Cayley and exponential retractions) to other general-purpose methods in the following scenarios:
\begin{itemize}
\item Scenario A (two asymptotically stable points): we take $\mathbb{I}=\operatorname{diag}(1,1.3,0.5)$, $\alpha=0.5$. Since the convergence to an isolated equilibrium point mitigates the differences in $M_{\text{final}}$ between methods, we choose a relatively short time interval, $t\in[0,30]$, so that simulations stop before they are too close to the limiting point. We plot with $\Omega_0=(0, 0.05, 1)$, $M_0=\mathbb{I}\Omega_0$, for a range of time step values $h$. 
\item Scenario B: the same as above, but with $t\in[0,250]$, in order to study the behavior near the limiting point.
\item Scenario C (infinite asymptotically stable points along a great circle): we take $\mathbb{I}=\operatorname{diag}(1,1,0.5)$, $\alpha=0.5$, $\Omega_0=(0, 0.05, 1)$, $M_0=\mathbb{I}\Omega_0$, $t\in[0,250]$, for a range of time step values $h$. The methods converge to the great circle but not necessarily to the same point.
\end{itemize}

The resulting trajectories for Scenarios B and C, using DDB with the Cayley retraction, are depicted in Figure~\ref{fig:spheresBC}.

\begin{figure}[h!]
    \centering
    \includegraphics[width=0.48\linewidth]{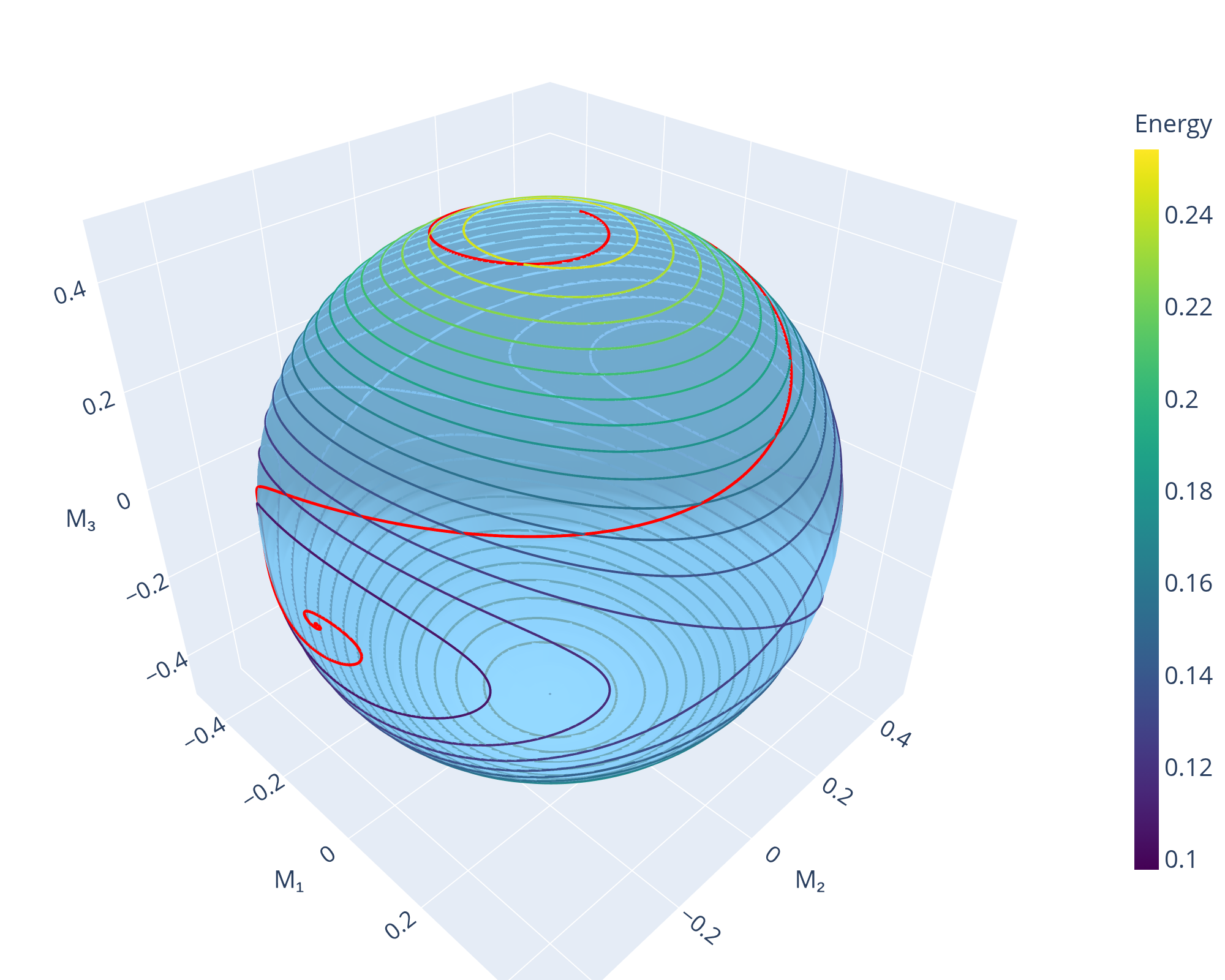}
    \hfill
    \includegraphics[width=0.48\linewidth]{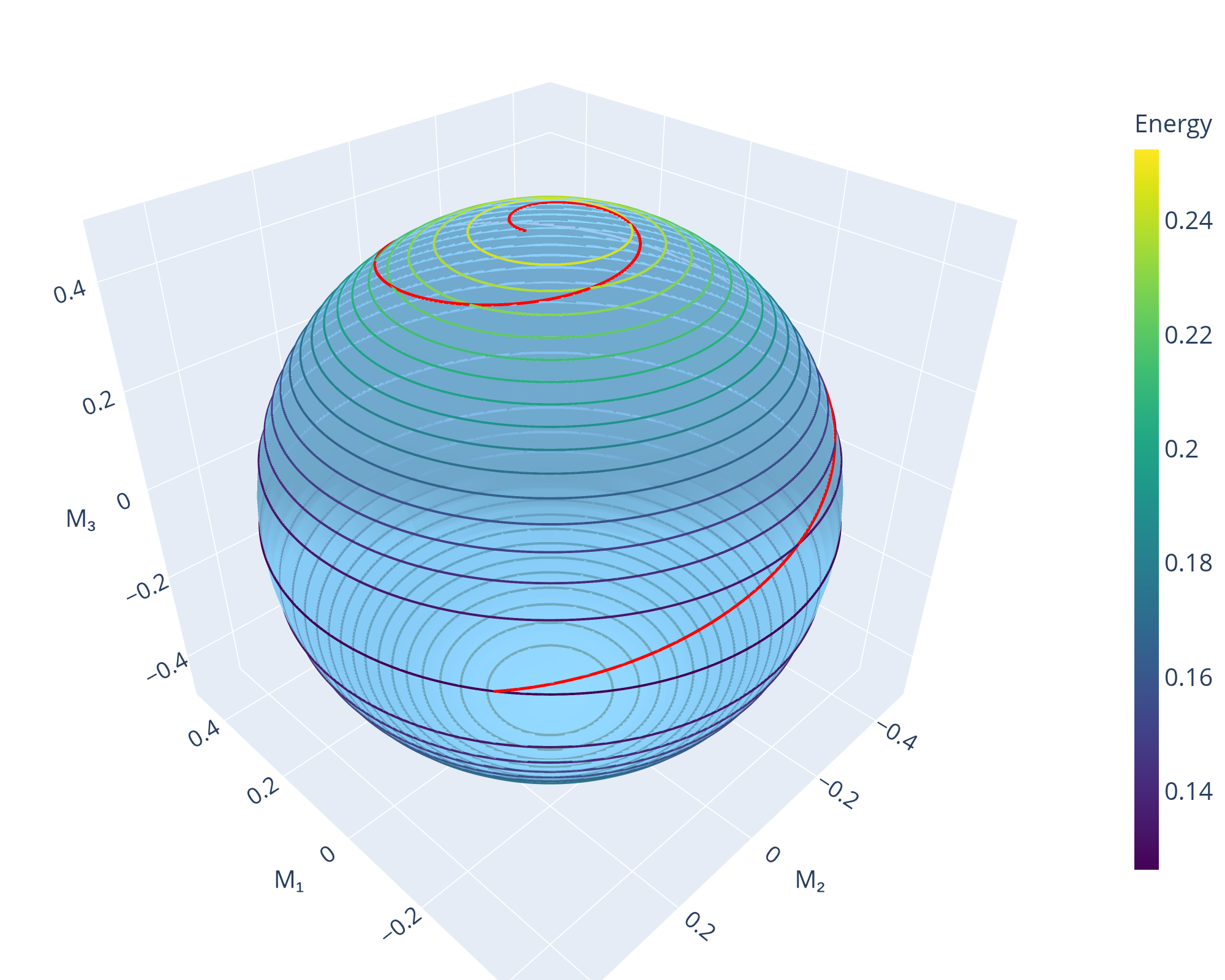}
    \caption{Trajectories for $M_k$, starting near the north pole and approaching the minimum energy sets asymptotically, for Scenarios B ($I_2>I_1>I_3$, left) and C ($I_1=I_2>I_3$, right).}
    \label{fig:spheresBC}
\end{figure}

The methods to be compared are Runge-Kutta 4th order, 3-stage Lobatto IIIC, DDB with $\tau=\operatorname{cay}$, and DDB with $\tau=\exp$. For the latter case we truncate the expansion~\eqref{eq:dexpinvxy} to order 2; higher orders yield essentially the same results.
As a reference solution, we use SciPy's built-in DOP853 method (8th order), with tight tolerances.

Applying each method to Scenario A using $h=0.1$, we compute the evolution of the Casimir $C(M_k)$ and plot it in Figure~\ref{fig:A-casimir_all_methods}. The DDB methods preserve the Casimir exactly, as expected. In Figure~\ref{fig:A-error_in_MN} we plot the error $\|M_N-M_N^{\text{ref}}\|$ between the momentum $M_N$ at the final step $N$ and the one corresponding to the reference solution, for a range of $h$ values between $1$ and $10^{-3}$.

Scenario B explores the asymptotic behavior for the case of two minimum energy points on $\mathcal{S}$. Figure~\ref{fig:B-distance_to_limit_point_all_methods} compares $M_N$ with the limit points, which are exactly $\pm \|M_0\|(0,1,0)$ as discussed above. In addition, Figure~\ref{fig:B-momentum_difference_norms} shows that the momentum of the DDB methods is eventually closer to the reference solution than the other methods (contrast with Figure~\ref{fig:A-error_in_MN}). This is due to the Casimir preservation property of the DDB methods and the asymptotically stable character of the equilibrium point. Both figures use $h=0.1$.

Scenario C shows the asymptotic behavior for the case where $\mathcal{M}$ is a great circle on $\mathcal{S}$, in this case on the plane corresponding to $I_1,I_2$. We compute the distance from $M_N$ to $\mathcal{M}$ and also to the plane (Figure~\ref{fig:C-distance_to_great_circle_all_methods}). Measuring the distance to the plane is more forgiving towards non--Casimir-preserving methods.

\begin{figure}[h!]
    \centering
    \includegraphics[width=0.8\linewidth]{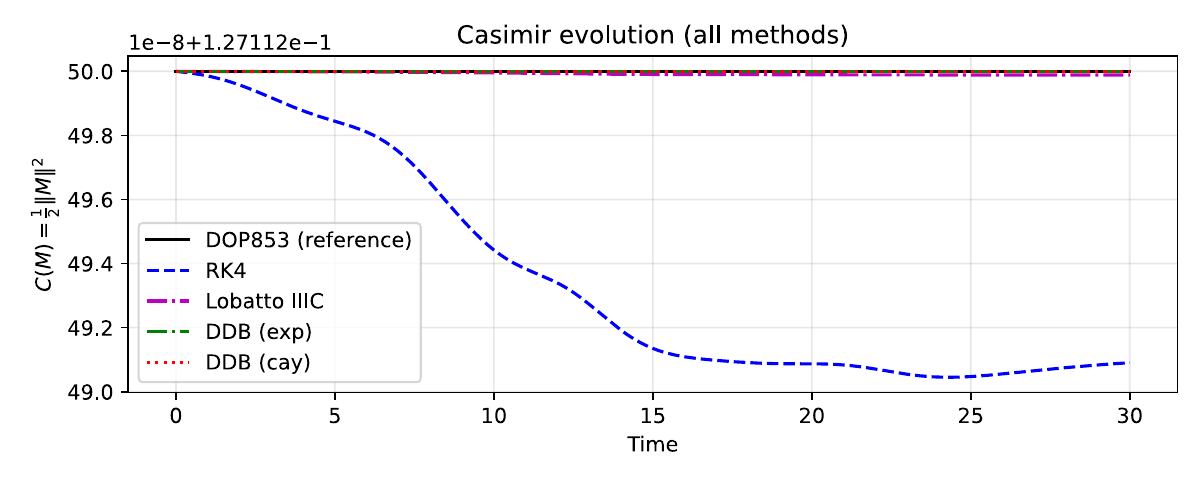}
    \includegraphics[width=0.8\linewidth]{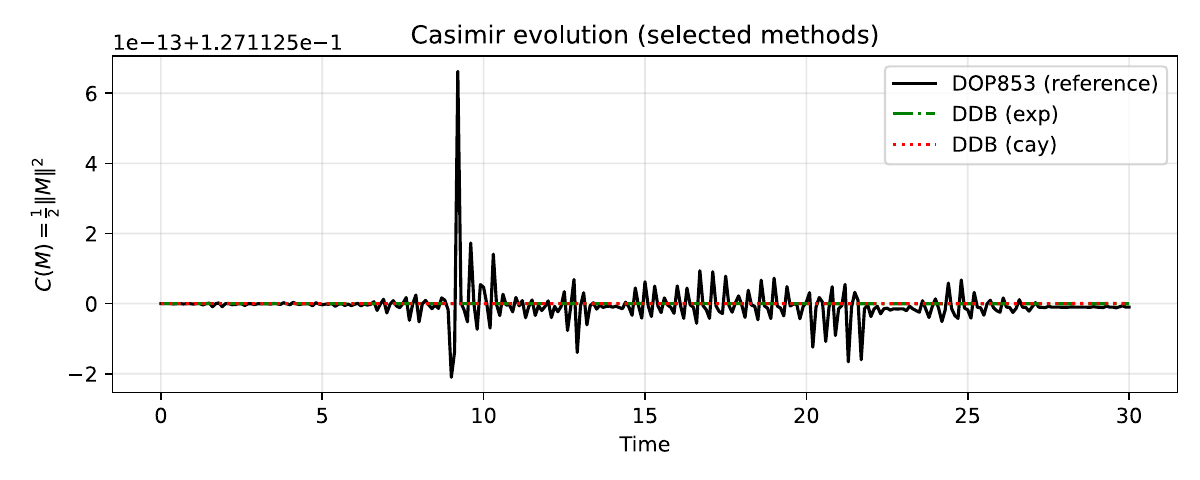}
    \caption{Casimir values for all methods and for DDB and reference methods only (Scenario A).}
    \label{fig:A-casimir_all_methods}
\end{figure}

\begin{figure}[h!]
    \centering
    \includegraphics[width=0.8\linewidth]{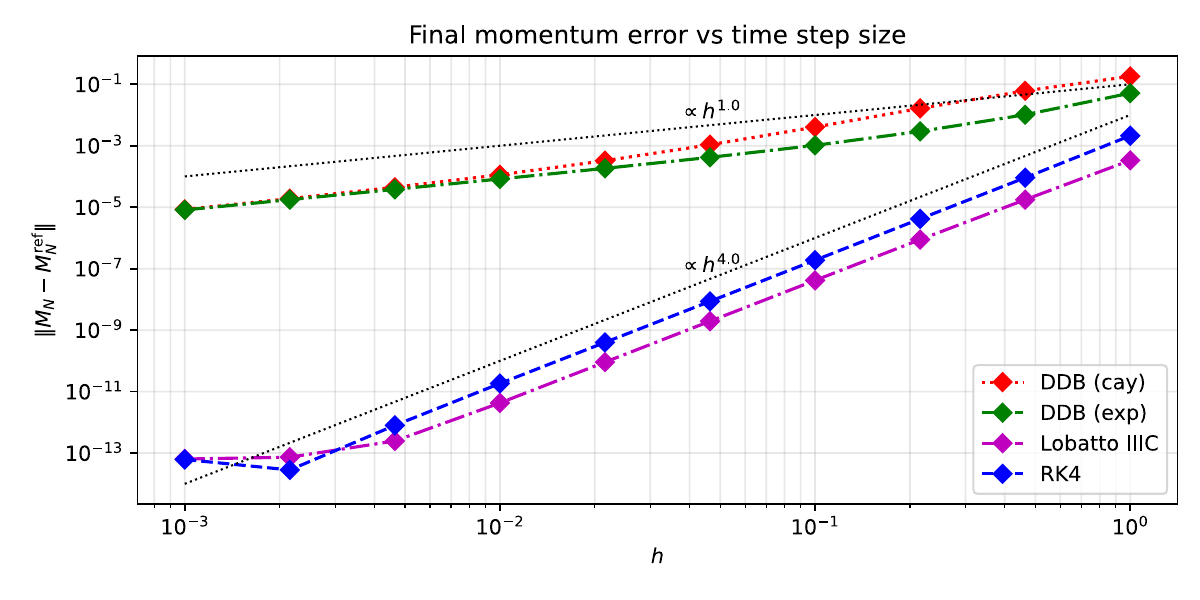}
    \caption{Error in final momentum $M_N$ (Scenario A).}
    \label{fig:A-error_in_MN}
\end{figure}

\begin{figure}[h!]
    \centering
    \includegraphics[width=0.8\linewidth]{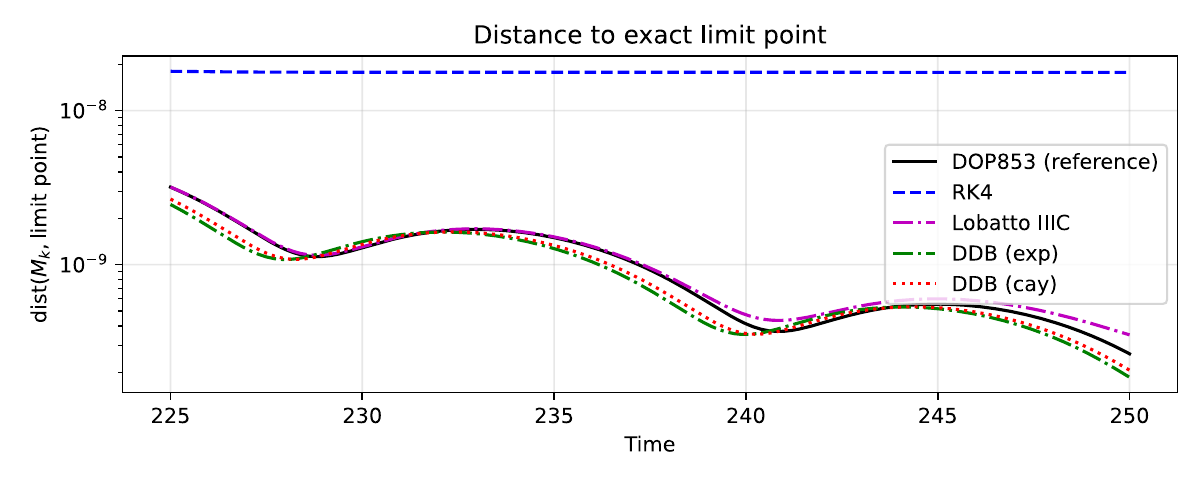}
    \caption{Distance of the discrete trajectory to the exact limit point during the last part of the simulation (Scenario B).}
    \label{fig:B-distance_to_limit_point_all_methods}
\end{figure}

\begin{figure}[h!]
    \centering
    \includegraphics[width=0.8\linewidth]{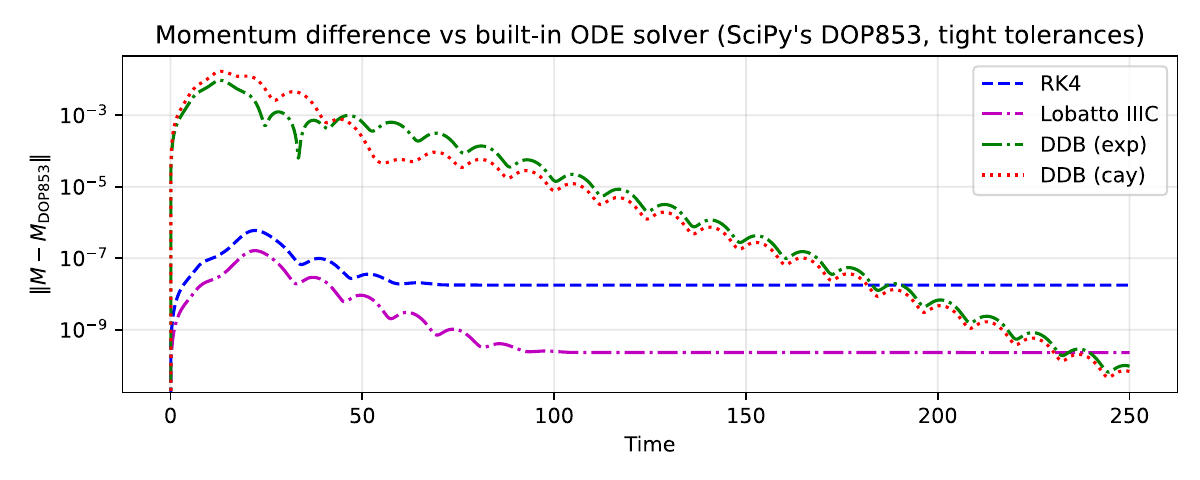}
    \caption{Momentum difference w.r.t. the reference trajectory (Scenario B).}
    \label{fig:B-momentum_difference_norms}
\end{figure}

\begin{figure}
    \centering
    \includegraphics[width=0.8\linewidth]{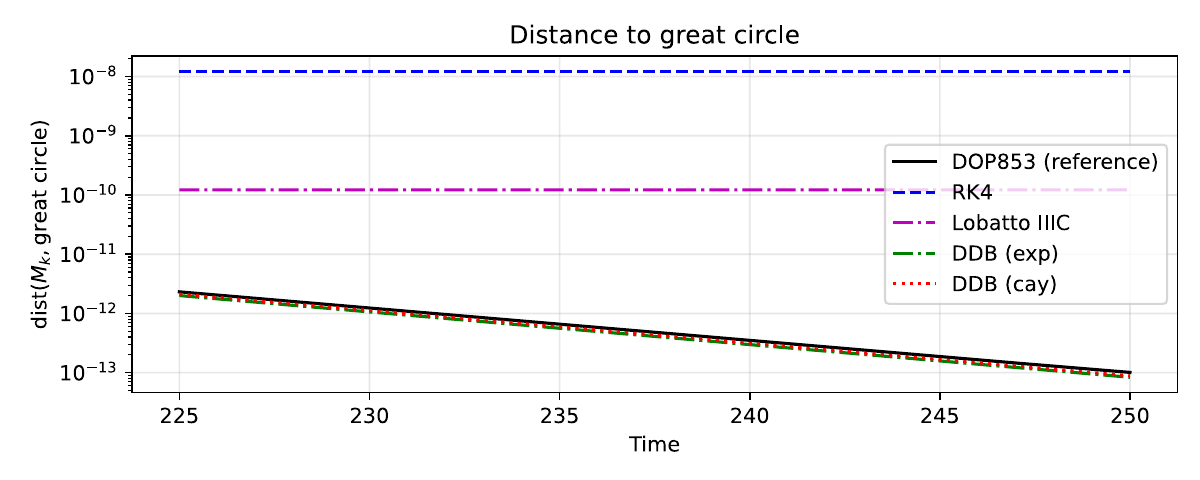}
    \includegraphics[width=0.8\linewidth]{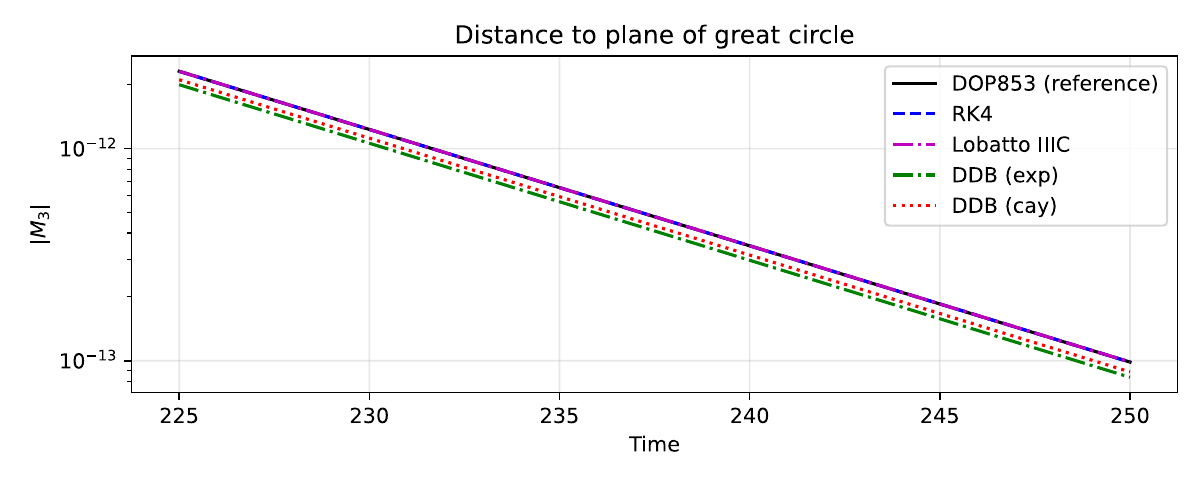}
    \caption{Distance to the great circle $\mathcal{M}$ of minimum energy points (top) and to the plane containing $\mathcal{M}$ (bottom), during the last part of the simulation (Scenario C).}
    \label{fig:C-distance_to_great_circle_all_methods}
\end{figure}

Finally, the execution time for each algorithm, for Scenario A and $h=0.1$, was as follows.

\begin{table}[h!]
    \centering
    \begin{tabular}{l|c|c|c|c|c}
     \textbf{Integrator} & 
       DOP853 (ref) &
  RK4                &
  Lobatto IIIC        &
  DDB (exp)           &
  DDB (cay)           \\ \hline 
     \textbf{Runtime (s)} & 
     0.207 &
   0.162 &
  1.616&
   0.107&
   0.108
    \end{tabular}
    \caption{Execution time for all methods}
    \label{tab:exectime}
\end{table}

An animation is available at \url{https://youtu.be/pxSJdQMxYQM} as supplementary material. It shows essentially the two behaviors illustrated in Figure \ref{fig:spheresBC}, except that the initial condition is moved closer to the north pole and the final time is $150$. 

\clearpage

\section{Conclusions and future work}

In this paper, we have developed a new geometric integrator for double-bracket dissipation, demonstrating excellent performance across a range of numerical simulations, including comparisons with higher-order methods.

The methods developed in this paper can be  extended to the general setting of Lie algebroids, whose duals are equipped with a linear Poisson structure. The goal is to preserve the associated symplectic foliation while allowing for energy dissipation. This framework includes mechanical systems arising from symmetry reduction.

In future work, we will derive numerical methods for the case of ideal fluids, where the setting is considerably richer and the double-bracket construction can be used to compute solutions that preserve the Casimir invariants of Euler’s fluid equations while dissipating energy (see \cite{bloch1996euler}). The main idea is to adapt the approach in \cite{Pavlov} to the methods developed in this paper, incorporating double-bracket dissipation through an approximation of the group of volume-preserving diffeomorphisms by a finite-dimensional Lie group.

\paragraph{Acknowledgments:} A.B. supported in part by NSF grant  DMS-2103026, and AFOSR grants FA
9550-22-1-0215 and FA 9550-23-1-0400. DMdD  acknowledges financial support from the Spanish Ministry of Science and Innovation under grants PID2022-137909NB-C21, PCI2024-155047-2 and  from the Severo Ochoa Programme for Centres of Excellence in R\&D (CEX2023-001347-S). SF acknowledges support from PGI-UNS 24/L124, PICT 2019-00196 Agencia I+D+i, FONCyT, Argentina, and the Severo Ochoa Programme.

% \bibliography{references}

\begin{thebibliography}{22}
\providecommand{\natexlab}[1]{#1}
\providecommand{\url}[1]{\texttt{#1}}
\expandafter\ifx\csname urlstyle\endcsname\relax
  \providecommand{\doi}[1]{doi: #1}\else
  \providecommand{\doi}{doi: \begingroup \urlstyle{rm}\Url}\fi

\bibitem[Abraham and Marsden(1978)]{Abraham1978}
Ralph Abraham and Jerrold~E. Marsden.
\newblock \emph{Foundations of Mechanics}.
\newblock {AMS Chelsea Publishing}, {Redwood City, CA}, 2 edition, 1978.

\bibitem[Blanes and Casas(2016)]{blanes}
Sergio Blanes and Fernando Casas.
\newblock \emph{A concise introduction to geometric numerical integration}.
\newblock Monographs and Research Notes in Mathematics. CRC Press, Boca Raton,
  FL, 2016.
\newblock ISBN 978-1-4822-6342-8.

\bibitem[Bloch et~al.(1996)Bloch, Krishnaprasad, Marsden, and
  Ratiu]{bloch1996euler}
Anthony Bloch, PS~Krishnaprasad, Jerrold~E Marsden, and Tudor~S Ratiu.
\newblock The {E}uler-{P}oincar{\'e} equations and double bracket dissipation.
\newblock \emph{Communications in mathematical physics}, 175\penalty0
  (1):\penalty0 1--42, 1996.

\bibitem[Bloch et~al.(2024)Bloch, Farr\'e{}~Puiggal\'i, and de~Diego]{BFM2025}
Anthony Bloch, Marta Farr\'e{}~Puiggal\'i, and David~Mart\'in de~Diego.
\newblock Metriplectic {E}uler-{P}oincar\'e{} equations: smooth and discrete
  dynamics.
\newblock \emph{Commun. Anal. Mech.}, 16\penalty0 (4):\penalty0 910--927, 2024.
\newblock ISSN 2836-3310.
\newblock \doi{10.3934/cam.2024040}.
\newblock URL \url{https://doi.org/10.3934/cam.2024040}.

\bibitem[Bloch et~al.(2002)Bloch, Crouch, Marsden, and
  Ratiu]{bloch2002symmetric}
Anthony~M Bloch, Peter~E Crouch, Jerrold~E Marsden, and Tudor~S Ratiu.
\newblock The symmetric representation of the rigid body equations and their
  discretization.
\newblock \emph{Nonlinearity}, 15\penalty0 (4):\penalty0 1309--1341, 2002.

\bibitem[Bloch et~al.(2013)Bloch, Morrison, and Ratiu]{Bloch-Morrison-Ratiu}
Anthony~M. Bloch, Philip~J. Morrison, and Tudor~S. Ratiu.
\newblock Gradient flows in the normal and {K}\"{a}hler metrics and triple
  bracket generated metriplectic systems.
\newblock In \emph{Recent trends in dynamical systems}, volume~35 of
  \emph{Springer Proc. Math. Stat.}, pages 371--415. Springer, Basel, 2013.
\newblock ISBN 978-3-0348-0451-6; 978-3-0348-0450-9.
\newblock \doi{10.1007/978-3-0348-0451-6\_15}.
\newblock URL \url{https://doi.org/10.1007/978-3-0348-0451-6\_15}.

\bibitem[Bou-Rabee and Marsden(2009)]{Bou-Rabee2009}
Nawaf Bou-Rabee and Jerrold~E. Marsden.
\newblock Hamilton--pontryagin integrators on {L}ie groups part i: Introduction
  and structure-preserving properties.
\newblock \emph{Foundations of Computational Mathematics}, 9\penalty0
  (2):\penalty0 197--219, 2009.
\newblock ISSN 1615-3383.
\newblock \doi{10.1007/s10208-008-9030-4}.
\newblock URL \url{https://doi.org/10.1007/s10208-008-9030-4}.

\bibitem[Campos et~al.(2025)Campos, Martín~de Diego, and
  Torrente]{CamposMartinTorrente2025preprint}
Cédric~M. Campos, David Martín~de Diego, and José Torrente.
\newblock Momentum-based gradient descent methods for {L}ie groups.
\newblock \emph{arXiv preprint arXiv:2404.09363}, 2025.

\bibitem[Hairer et~al.(2010)Hairer, Lubich, and Wanner]{hairer}
E.~Hairer, C.~Lubich, and G.~Wanner.
\newblock \emph{Geometric numerical integration}, volume~31 of \emph{Springer
  Series in Computational Mathematics}.
\newblock Springer, Heidelberg, 2010.
\newblock ISBN 978-3-642-05157-9.
\newblock Structure-preserving algorithms for ordinary differential equations,
  Reprint of the second (2006) edition.

\bibitem[Holm(2011{\natexlab{a}})]{holm-book1}
Darryl~D. Holm.
\newblock \emph{Geometric mechanics. {P}art {I}}.
\newblock Imperial College Press, London, second edition, 2011{\natexlab{a}}.
\newblock ISBN 978-1-84816-775-9; 1-84816-775-X.
\newblock Dynamics and symmetry.

\bibitem[Holm(2011{\natexlab{b}})]{holm-book2}
Darryl~D. Holm.
\newblock \emph{Geometric mechanics. {P}art {II}. {R}otating, translating and
  rolling}.
\newblock Imperial College Press, London, second edition, 2011{\natexlab{b}}.
\newblock ISBN 978-1-84816-778-0; 1-84816-778-4.
\newblock \doi{10.1142/p802}.
\newblock URL \url{https://doi.org/10.1142/p802}.

\bibitem[Iserles et~al.(2000)Iserles, Munthe-Kaas, Nørsett, and
  Zanna]{Iserles_Munthe-Kaas_Nørsett_Zanna_2000}
Arieh Iserles, Hans~Z. Munthe-Kaas, Syvert~P. Nørsett, and Antonella Zanna.
\newblock Lie-group methods.
\newblock \emph{Acta Numerica}, 9:\penalty0 215–365, 2000.
\newblock \doi{10.1017/S0962492900002154}.

\bibitem[LaSalle(1960)]{LaSalle1960}
J.~P. LaSalle.
\newblock Some extensions of {L}iapunov's second method.
\newblock \emph{IRE Trans.}, CT-7:\penalty0 520--527, 1960.

\bibitem[Marrero et~al.(2006)Marrero, Mart\'in~de Diego, and Mart\'inez]{MMM}
Juan~C. Marrero, David Mart\'in~de Diego, and Eduardo Mart\'inez.
\newblock Discrete {L}agrangian and {H}amiltonian mechanics on {L}ie groupoids.
\newblock \emph{Nonlinearity}, 19\penalty0 (6):\penalty0 1313--1348, 2006.
\newblock ISSN 0951-7715,1361-6544.
\newblock \doi{10.1088/0951-7715/19/6/006}.
\newblock URL \url{https://doi.org/10.1088/0951-7715/19/6/006}.

\bibitem[Marsden and West(2001)]{marsden-west}
J.~E. Marsden and M.~West.
\newblock Discrete mechanics and variational integrators.
\newblock \emph{Acta Numer.}, 10:\penalty0 357--514, 2001.
\newblock ISSN 0962-4929.
\newblock \doi{10.1017/S096249290100006X}.
\newblock URL \url{http://dx.doi.org/10.1017/S096249290100006X}.

\bibitem[Marsden and Ratiu(1986)]{marsden-ratiu}
Jerrold~E. Marsden and Tudor Ratiu.
\newblock Reduction of {P}oisson manifolds.
\newblock \emph{Lett. Math. Phys.}, 11\penalty0 (2):\penalty0 161--169, 1986.
\newblock ISSN 0377-9017.
\newblock \doi{10.1007/BF00398428}.
\newblock URL \url{https://doi.org/10.1007/BF00398428}.

\bibitem[Marsden and Ratiu(2013)]{marsden2013introduction}
Jerrold~E Marsden and Tudor~S Ratiu.
\newblock \emph{Introduction to mechanics and symmetry: a basic exposition of
  classical mechanical systems}, volume~17.
\newblock Springer Science \& Business Media, 2013.

\bibitem[Marsden et~al.(1999)Marsden, Pekarsky, and Shkoller]{MaPeSh99}
Jerrold~E Marsden, Sergey Pekarsky, and Steve Shkoller.
\newblock Discrete {E}uler-{P}oincar\'e{} and {L}ie-{P}oisson equations.
\newblock \emph{Nonlinearity}, 12\penalty0 (6):\penalty0 1647--1662, 1999.
\newblock ISSN 0951-7715,1361-6544.
\newblock \doi{10.1088/0951-7715/12/6/314}.
\newblock URL \url{https://doi.org/10.1088/0951-7715/12/6/314}.

\bibitem[Morrison(1986)]{morrison}
Philip~J. Morrison.
\newblock A paradigm for joined {H}amiltonian and dissipative systems.
\newblock \emph{Phys. D}, 18\penalty0 (1-3):\penalty0 410--419, 1986.
\newblock ISSN 0167-2789.
\newblock \doi{10.1016/0167-2789(86)90209-5}.
\newblock URL \url{https://doi.org/10.1016/0167-2789(86)90209-5}.
\newblock Solitons and coherent structures (Santa Barbara, Calif., 1985).

\bibitem[Moser and Veselov(1991)]{moser-veselov0138}
J\"urgen Moser and Alexander~P. Veselov.
\newblock Discrete versions of some classical integrable systems and
  factorization of matrix polynomials.
\newblock \emph{Comm. Math. Phys.}, 139\penalty0 (2):\penalty0 217--243, 1991.
\newblock ISSN 0010-3616,1432-0916.
\newblock URL \url{http://projecteuclid.org/euclid.cmp/1104203302}.

\bibitem[Pavlov et~al.(2011)Pavlov, Mullen, Tong, Kanso, Marsden, and
  Desbrun]{Pavlov}
D.~Pavlov, P.~Mullen, Y.~Tong, E.~Kanso, J.~E. Marsden, and M.~Desbrun.
\newblock Structure-preserving discretization of incompressible fluids.
\newblock \emph{Phys. D}, 240\penalty0 (6):\penalty0 443--458, 2011.
\newblock ISSN 0167-2789,1872-8022.
\newblock \doi{10.1016/j.physd.2010.10.012}.
\newblock URL
  \url{https://doi-org.csic.idm.oclc.org/10.1016/j.physd.2010.10.012}.

\bibitem[Sanz-Serna and Calvo(1994)]{serna}
J.~M. Sanz-Serna and M.~P. Calvo.
\newblock \emph{Numerical {H}amiltonian problems}, volume~7 of \emph{Applied
  Mathematics and Mathematical Computation}.
\newblock Chapman \& Hall, London, 1994.
\newblock ISBN 0-412-54290-0.

\end{thebibliography}

\end{document}